\documentclass{article}

\usepackage{lscape}
\usepackage[utf8]{inputenc} % allow utf-8 input
\usepackage[T1]{fontenc}    % use 8-bit T1 fonts
\usepackage{hyperref}       % hyperlinks
\usepackage{url}            % simple URL typesetting
\usepackage{booktabs}       % professional-quality tables
\usepackage{amssymb}
\usepackage{amsfonts}       % blackboard math symbols
\usepackage{nicefrac}       % compact symbols for 1/2, etc.
\usepackage{microtype}      % microtypography
\usepackage{xcolor}         % colors
\usepackage{amsmath}
\usepackage{amsthm}
\usepackage{algorithm}
\usepackage{algpseudocode}
\usepackage{graphicx,xcolor}
\usepackage{caption}
\usepackage{subcaption}
\usepackage{multirow}
\usepackage{geometry}
\usepackage{longtable}
\usepackage{xcolor}
\usepackage{lmodern}

\newtheorem{fact}{Fact}	
\newtheorem{lemma}{Lemma}
\newtheorem{theorem}{Theorem}

\newcommand{\R}{\mathbb{R}}
\newcommand{\Diag}{\mathrm{Diag}}
\newcommand{\diag}{\mathrm{diag}}

\def\norm#1{\left\|#1\right\|}

\begin{document}

\title{A Sparse Smoothing Newton Method for Solving Discrete Optimal Transport Problems}

\author{Di Hou~\thanks{Department of Mathematics, National University of Singapore, Singapore. \url{dihou@u.nus.edu}}, Ling Liang~\thanks{(Corresponding author) Department of Mathematics,
  University of Maryland at College Park, USA. \url{liang.ling@u.nus.edu}}, and Kim-Chuan Toh~\thanks{Department of Mathematics and Institute of Operations Research and Analytics, National University of Singapore, Singapore. \url{mattohkc@nus.edu.sg}}}
\date{}  
 
\maketitle

\begin{abstract}
The discrete optimal transport (OT) problem, which offers an effective computational tool for comparing two discrete probability distributions, has recently attracted much attention and played essential roles in many modern applications. This paper proposes to solve the discrete OT problem by applying a squared smoothing Newton method via the Huber smoothing function for solving the corresponding KKT system directly. The proposed algorithm admits appealing convergence properties and is able to take advantage of the solution sparsity to greatly reduce computational costs. Moreover, the algorithm can be extended to solve problems with similar structures, including the Wasserstein barycenter (WB) problem with fixed supports. To verify the practical performance of the proposed method, we conduct extensive numerical experiments to solve a large set of discrete OT and WB benchmark problems. Our numerical results show that the proposed method is efficient compared to state-of-the-art linear programming (LP) solvers. Moreover, the proposed method consumes less memory than existing LP solvers, which demonstrates the potential usage of our algorithm for solving large-scale OT and WB problems.    
\end{abstract}


\section{Introduction} \label{sec:intro}
We are interested in the discrete optimal transport (OT) problem that takes the form of
\begin{equation}
    \label{eq:POT}
    \min_{X\in \mathbb{R}^{m\times n}}\quad \left\langle C, X \right\rangle \quad \mathrm{s.t.}\quad Xe_n = a, \; X^Te_m = b,\; X\geq 0, 
\end{equation}
where $C\in \mathbb{R}^{m\times n}_+ $ denotes a certain cost matrix, $a \in \mathbb{R}^m_+$ and $b \in \mathbb{R}^n_+$ denote two discrete probability distributions satisfying $e_m^T a =  e_n^T b = 1$, and $ e_m $ (resp. $ e_n $) denotes the vector of all ones in $ \mathbb{R}^m $ (resp. $ \mathbb{R}^n $). By the standard Lagrangian duality theory, the dual problem of \eqref{eq:POT} is given as the following maximization problem:
\begin{equation}\label{eq:DOT}
	\max_{f\in\mathbb{R}^m, g\in \mathbb{R}^n, Z\in\mathbb{R}^{m\times n}}\quad \left\langle a, f\right\rangle  + \left\langle b, g\right\rangle \quad \mathrm{s.t.}\quad fe_n^T + e_mg^T + Z = C,\; Z\geq 0. 
\end{equation}
Obviously, the feasible set for  problem \eqref{eq:POT} is nonempty and compact. Hence, the primal problem \eqref{eq:POT} has a finite optimal value that is attainable. Since the strong duality for linear programming always holds, it is clear that the following KKT optimality conditions for \eqref{eq:POT}--\eqref{eq:DOT}:
\begin{equation}\label{eq:KKT}
	Xe_n = a, \quad  X^Te_m = b,\quad fe_n^T + e_mg^T + Z = C,\quad X\geq 0,\; Z\geq 0,\; \left\langle X, Z \right\rangle = 0
\end{equation}
admits at least one solution. Note that the last three conditions in \eqref{eq:KKT} can be reformulated as the nonsmooth system $ X - \Pi_+(X-\sigma Z) = 0 $, where $\Pi_+(\cdot)$ is the projection operation onto the nonnegative orthant $ \mathbb{R}^{m\times n}_+ $ and $\sigma > 0$ is any positive constant.

The OT problem has a large number of important applications, due partly to the essential metric property of its optimal solution. In particular, the optimal objective value of the OT problem defines a distance (i.e., the Wasserstein distance) between two probability distributions if the cost matrix $C$ is chosen based on some distances; see for instance \cite{fuDetectingPhishingWeb2006, graumanFastContourMatching2004, kendalQuantifyingPlantColour2013,   peyre2019computational, rubnerEarthMoverDistance2000, santambrogio2015optimal, villani2009optimal} and references therein for fundamental properties of OT problems together with many important applications in a wide variety of fields including imaging science, engineering, statistics and machine learning, management, finance, and beyond. {For example, a representative} {application of the Wasserstein distance is the Wasserstein barycenter {(WB) problem}, which aims to compute the mean of a set of discrete probability distributions under the Wasserstein distance \cite{agueh2011barycenters}. It has shown promising performance in many fields such as machine learning \cite{cuturi2014fast, peyre2019computational, ye2014scaling}, image processing \cite{rabin2012wasserstein}, and economics \cite{chiappori2010hedonic, galichon2018optimal}. For a set of discrete probability distributions with finite support points, a WB problem with its support points pre-specified can be reformulated as a {linear programming (LP)} problem \cite{anderes2016discrete} {that shares a similar structure as the OT problem; See section \ref{sec:wbp} for the detailed description of the WB problem with fixed supports}.}

As real-world applications typically produce large-scale OT problems, computing the optimal transportation plans efficiently and accurately at scale has received growing attention. Indeed, the last few years have seen a blossoming of interest in developing efficient solution methods for solving OT problems. First, if we treat the OT problem as a special case of a general {LP} problem, it is clear that any solver for LP can be applied. In particular, it is commonly accepted that~\footnote{Based on the benchmark conducted by Hans D. Mittelmann, see \url{http://plato.asu.edu/}.} primal-dual interior point methods (IPMs) \cite{wright1997primal} and simplex methods \cite{dantzig2016linear}, such as the highly optimized commercial solvers GUROBI, MOSEK and CPLEX, and the open-source solver HiGHS, are the most efficient and robust solvers for general LPs. Moreover, when an LP problem (including the OT problem) has many more variables than the number of linear constraints, \cite{li2020asymptotically} provides strong evidence that the dual-based inexact augmented Lagrangian method is also highly efficient and consumes less memory than that of IPMs. On the other hand, it is known from the literature (see e.g., \cite[Chapter 3]{peyre2019computational}) that the OT problem can be viewed as an optimal network flow problem built upon a bipartite graph having $(m+n)$ nodes for which $a$ and $b$ are considered as the sources and sinks, respectively. Then, the network simplex method \cite{orlin1997polynomial} can also be applied to solve the OT problem exactly. In fact, as demonstrated in \cite{schrieber2016dotmark}, the network simplex method (implemented in CPLEX) and its variant (such as the transportation simplex method \cite{luenberger1984linear}) are very efficient and robust. Last, for applications for which accurate solutions are not required, entropic regularized algorithms (see e.g., the Sinkhorn algorithm \cite{cuturi2013sinkhorn} and its stabilized variants \cite{schmitzer2019stabilized}), entropic and/or Bregman proximal point methods (see e.g., \cite{chu2020efficient, yang2022bregman}) and multiscale approaches \cite{liu2022multiscale,merigot2011multiscale, schmitzer2016sparse} are applicable and popular.

When it comes to the WB problem, the problem scale is usually much larger than that of the OT problem due to the {need to compute multiple transport plans for the given} discrete distributions. Thus, it is quite challenging to solve the {WB} problems by classical methods like IPMs or simplex methods. To overcome the computational challenges, one promising approach is to consider the entropic regularized problem \cite{benamou2015iterative, cuturi2014fast}, for which scalable Sinkhorn based algorithms can be developed for solving large-scale problems (see e.g., \cite{yang2021fast, ye2017fast, zhang2022efficient}). However, the entropic regularized problem typically can only produce a rough approximate solution to the original WB problem. To get better approximate solutions, it is necessary to choose a small regularization parameter, but this will generally cause some numerical instabilities and also inefficiency to the Sinkhorn based algorithms. Thus, in this paper, we aim to solve the OT and WB problems directly without introducing an entropic regularization.

In this paper, we focus on solving the KKT system \eqref{eq:KKT} directly by Newton-type methods. We know that the OT problem \eqref{eq:POT} has many more variables than the number of linear constraints. As a consequence, an optimal solution (not necessarily unique) of an OT problem is generally highly sparse. In fact, as an optimal network flow problem, the OT problem admits an optimal solution having at most $m+n$ nonzero entries (see e.g., \cite[Section 3.4]{peyre2019computational}). It is well-known that a primal-dual IPM applies the Newton method for solving a sequence of perturbed KKT systems. This approach is shown to be highly efficient and robust in general. However, as we will explain in section \ref{sec:preliminary}, the primal-dual IPM needs to solve a sequence of dense and highly ill-conditioned  linear systems which require excessive computational costs. Moreover, as the iterates are required to stay in the positive orthant, IPM is incapable of exploiting the solution sparsity when solving the OT problem. 

The main issue that we want to address in this paper is: How can we exploit the solution sparsity in designing a Newton-type method for solving the KKT system \eqref{eq:KKT} directly? One popular way is to {combine} an IPM with the column generation technique; See e.g., \cite{ye1992potential, tits2006constraint} for more details. However, the column generation technique is effective only if one is able to identify the correct active set efficiently, which requires sophisticated pricing strategies that are usually highly heuristic. Another approach is to apply the semismooth Newton method \cite{qi1993nonsmooth} to solve the original KKT system directly. It is shown that, under suitable conditions, the semismooth Newton (SSN) method has a local superlinear or even quadratic convergence rate. However, there exist some practical and theoretical issues when applying the SSN method, as we shall explain later in section \ref{sec:preliminary}. Among these issues, the most critical one is that the global convergence of SSN cannot be guaranteed due to the lack of a suitable merit function for line search. 

To address the question just raised, we propose a squared smoothing Newton method via the Huber function for solving OT problems. The proposed method admits appealing global convergence properties, and it is able to take advantage of the underlying sparsity structure of the optimal solution. In addition, the proposed method enjoys superlinear convergence to an accumulation point. Extensive numerical experiments conducted on solving OT problems generated from the DOTmark collection \cite{schrieber2016dotmark} show the promising practical performance of our proposed method when compared to highly optimized LP solvers and the network simplex method. Besides the successful application to OT problems, we further extend the method to solve WB problems with fixed supports, and provide an efficient way to solve the corresponding Newton system. Numerical results on some real image dataset indicate that the proposed method outperforms both the primal-dual IPM and dual simplex method implemented in the highly optimized commercial solver Gurobi.

The rest of this paper is organized as follows. In section \ref{sec:preliminary}, to motivate the development of our proposed algorithm, we briefly introduce the main ideas of the primal-dual IPM and the semismooth Newton method for solving OT problems, and explain some of the limitations of both methods. Then in section \ref{sec:sqsn}, we present our main algorithm, i.e., a squared smoothing Newton method via the Huber smoothing function, and conduct rigorous convergence analysis. In section \ref{sec:wbp}, we extend our algorithm to solve WB problems with fixed supports, and explain how to solve the underlying Newton equations efficiently. To evaluate the practical performance of the proposed method, we conduct numerical experiments on numerous OT and WB problems with different sizes in section \ref{sec:numexp}. Finally, we conclude our paper in section \ref{sec:conclusions}.

\section{Preliminary} \label{sec:preliminary}

In this article, the following notations are used. For any finite dimensional Euclidean space $\mathbb{E}$  equipped with an inner product $\left\langle \cdot, \cdot\right\rangle$, we denote its induced norm by $\left\lVert\cdot\right\rVert$. {We denote the nonnegative orthant of $\mathbb{R}^n$ as $\mathbb{R}^n_+$.} For any positive integer $n$, $e_n\in \mathbb{R}^n$ denotes the vector of all ones and $I_n\in \mathbb{R}^{n\times n}$ denotes the identity matrix. Let $X\in \mathbb{R}^{m\times n}$, we use $\mathrm{vec}(X)$ to denote the vector in $ \mathbb{R}^{mn} $ obtained by stacking the columns of $X$. Conversely, $X = \mathrm{Mat}(x)$ is the unique matrix such that $x = \mathrm{vec}(X)$. For given matrices $A$ and $B$, $A\otimes B$ denotes the Kronecker product of $A$ and $B$. Additionally, for any matrix $X$ such that $AXB$ is well-defined, it holds that $\mathrm{vec}(AXB) = (B^T\otimes A)\mathrm{vec}(X)$. We also use $\circ$ to denote the element-wise multiplication operator of two vectors/matrices of the same size. Let $x$ be any vector, $\mathrm{Diag}(x)$ denotes the diagonal matrix with $x$ on its diagonal. Conversely, for any square matrix $X$, $\mathrm{diag}(X)$ is the vector consisting of the diagonal entries of $X$. For any vector $x$, the projection of $x$ onto $\mathbb{R}^n_+$ is denoted by $\Pi_+(x) := \mathrm{argmin}\{\left\lVert z - x\right\rVert^2\;:\; z\geq 0\}=\max\{x,0\}$, where the $\max$ operation is applied elementwise on $x$.

In this section, we shall briefly discuss the primal-dual IPM and the semismooth Newton (SSN) method, which are both Newton-type methods for handling the KKT system \eqref{eq:KKT}. We also mention some theoretical and practical issues that we may face when applying these two methods to solving the OT problem. 

We start by introducing some notation that would simplify our presentation. Let $x:= \mathrm{vec}(X) \in \mathbb{R}^{mn}$. By using the properties of the Kronecker product, we see that 
\[
	Xe_n = (e_n^T\otimes I_m)x,\quad X^Te_m = (I_n\otimes e_m^T)x.
\]
Denote $z:= \mathrm{vec}(Z) \in \mathbb{R}^{mn}$, $ c:= \mathrm{vec}(C) \in \mathbb{R}^{mn} $, $ d:= (a;b) \in \mathbb{R}^{m+n} $, $ y:= (f;g) \in \mathbb{R}^{m+n} $ and $ A:= \begin{pmatrix}
	e_n^T\otimes I_m \\
	I_n\otimes e_m^T 
\end{pmatrix} $. Then we can rewrite the KKT systems \eqref{eq:KKT} as follows:
\begin{equation}\label{eq:kkt-new}
	Ax = d, \quad A^Ty + z = c,\quad x\circ  z = 0,\; x\geq 0,\; z\geq 0. 
\end{equation}

%%%%%%%%%%%%%%%%%%%%%%%%%%
\subsection{Primal-dual interior point methods} \label{subsec:IPM}

Recall that a primal-dual IPM solves the KKT system \eqref{eq:KKT} {indirectly} by solving a sequence of perturbed KKT systems of the form:
\begin{equation}\label{eq:perb-kkt}
	Ax = d, \quad A^Ty + z = c,\quad x\circ  z = \mu e_{mn},\; x\geq 0,\; z\geq 0,
\end{equation}
via the Newton method, where $\mu > 0$ is the barrier parameter that is driven to zero. At each IPM iteration, the following Newton equation is solved:
\begin{equation}
	\label{eq:IPM-linsys}
	\begin{pmatrix}
		A & 0 & 0 \\
		0 & A^T & I_{mn} \\
		\mathrm{Diag}(z) & 0 & \mathrm{Diag}(x)
	\end{pmatrix}
\begin{pmatrix}
	\Delta x \\ \Delta y \\ \Delta z
\end{pmatrix}
=
\begin{pmatrix}
	r_p := d - Ax \\
	r_d := c - A^Ty - z \\
	r_c := \gamma \mu e_{mn} - x\circ z
\end{pmatrix},
\end{equation}
where $(x, y, z) \in \mathbb{R}^{mn}_{++}\times \mathbb{R}^{m+n} \times \mathbb{R}^{m n}_{++}$ is the current primal-dual iterate, $ (\Delta x, \Delta y, \Delta z) \in \mathbb{R}^{m n}\times \mathbb{R}^{m+n} \times \mathbb{R}^{m n}$ is the Newton direction, $\mu$ is set to be $\langle x,z\rangle/(mn)$, and $\gamma \in [0,1]$ is a given parameter. To reduce the computational cost, the large-scale system of $2mn+m+n$ linear equations is typically reduced to the so-called normal system of $m+n$ equations in $\Delta y$ through the block elimination of $\Delta x$ and $\Delta z$. Specifically, we solve the normal equation:
\begin{equation}
	\label{eq:normal-system-ipm}
	A\Theta A^T \Delta y = r_p + A\Theta (r_d - \mathrm{Diag}(x)^{-1}r_c),
\end{equation}
where $\Theta:=\mathrm{Diag}(\theta)\in \mathbb{R}^{mn\times mn}$ is a diagonal matrix whose diagonal entries are $\theta_i = x_i/z_i$, for $i = 1,\dots, mn$. After obtaining $\Delta y$ by solving \eqref{eq:normal-system-ipm}, we can compute $\Delta x$ and $\Delta z$ accordingly. To solve \eqref{eq:normal-system-ipm} efficiently, the following structure of $A\Theta A^T$ is useful.

\begin{fact}\label{fact-ADAT}
	For any diagonal matrix $\Theta := \mathrm{Diag}(\theta)\in \mathbb{R}^{mn\times mn}$, let $V:= \mathrm{Mat}(\theta) \in \mathbb{R}^{m\times n}$, then it holds that 
	\[
	A\Theta A^T = \begin{pmatrix}
		\mathrm{Diag}(Ve_n) & V \\
		V^T & \mathrm{Diag}(V^Te_m)
	\end{pmatrix}.
	\]
\end{fact}

As we can see, the iterates $x$ and $z$ generated by an IPM should be positive because of the third equation in \eqref{eq:perb-kkt}. As a consequence, the matrix $V = \mathrm{Mat}(\theta)$ is dense and $A\Theta A^T$ contains two fully dense blocks. Therefore, factorizing $ A\Theta  A^T$ in \eqref{eq:normal-system-ipm} can be very expensive when $m+n$ is large. As a result, solving the normal system by an iterative solver is the only viable option. Moreover, it is obvious that $ A\Theta A^T$ is always singular, since the vector $(e_m; -e_n) \in \mathbb{R}^{m+n}$ is an eigenvector for $A\Theta  A^T$ corresponding to the zero eigenvalue. Note that the singularity of the coefficient matrix may not be an issue since one can remove the last row of $ A$ in a prepossessing phase, which results in a normal equation having a coefficient matrix of the form 
\[
\begin{pmatrix}
	\mathrm{Diag}(Ve_n) & \hat V \\
	\hat V^T & \mathrm{Diag}(\hat V^Te_m)
\end{pmatrix} \in \mathbb{R}^{(m+n-1)\times(m+n-1)} ,
\]
where $\hat V \in \mathbb{R}^{m\times (n-1)}$ is formed from the first $n-1$ columns of $V$. Then, by \cite[Lemma 2.3]{hu2021semismooth}, we see that the resulting coefficient matrix is nonsingular. However, the issue of ill-conditioning could still be a critical 
difficulty that goes against an iterative solver for fast convergence. In view of the above discussions, IPMs may not perform as efficiently as one would expect when solving large-scale OT problems. 

\subsection{The semismooth Newton method} \label{subsec:ssn}
In this subsection, we shall discuss the SSN method. We first revisit the concept of semismoothness which was originally introduced in \cite{mifflin1977semismooth} for functionals and later extended for vector-valued functions in \cite{qi1993nonsmooth}. Let $\mathbb{X}$ and $\mathbb{Y}$ be two finite dimensional Euclidean spaces and $F:\mathbb{X}\rightarrow \mathbb{Y}$ be a locally Lipschitz continuous function. According to Rademacher's Theorem \cite{rademacher1919partielle}, $F$ is F-differentiable almost everywhere, and the Clarke's generalized Jacobian \cite{clarke1990optimization} of $F$ at $x\in \mathbb{X}$, denoted by $\partial F(x)$, is well-defined and it holds that
\[
\partial F(x) := \mathrm{conv}(\partial_BF(x)), \quad \partial_BF(x):= \left\{ \lim_{D_F \ni z \rightarrow x} F'(z) \right\},\quad  \forall\; x\in \mathbb{X},
\]
where ``$\mathrm{conv}$'' denotes the convex hull operation and $D_F \subset \mathbb{X}$ is the set of points at which $F$ is F-differentiable. 

If $F$ is directionally differentiable at $x\in \mathbb{X}$, i.e., the limit
\[
F'(x;d) := \lim_{t\downarrow 0} \frac{F(x + td) - F(x)}{t} 
\]  
exists for all $d\in \mathbb{X}$, and 
$
\left\lVert F(x + d) - F(x) - V d\right\rVert = o(\left\lVert d\right\rVert),\; \forall V\in \partial F(x + d),\; \forall d\rightarrow 0,
$
we say that $F$ is semismooth at $x$. If $F$ is semismooth at $ x\in \mathbb{X} $ and it holds that 
\[
\left\lVert F(x + d) - F(x) - V d\right\rVert = O(\left\lVert d\right\rVert^2),\quad \forall V \in \partial F(x + d),\; \forall d\rightarrow 0,
\]
then $F$ is said to be strongly semismooth at $x$. Moreover, $F$ is said to be (strongly) semismooth everywhere on $\mathbb{X}$, if $F$ is (strongly) semismooth at any point $x\in \mathbb{X}$. 

Convex functions, smooth functions, and piecewise linear functions are examples of semismooth functions. In particular, one can verify that the projection operator onto the nonnegative orthant is strongly semismooth everywhere. Hence, the  KKT mapping $F(x, y, z)$ of the following equation
is also strongly semismooth everywhere:
\begin{equation}\label{eq-kkt-mapping}
 0\;=\;	F(x, y, z):= 
	\begin{pmatrix}
		Ax - d \\
		-A^Ty - z + c \\
		x - \Pi_+(x - \sigma z)
	\end{pmatrix},\quad (x, y, z)\in \mathbb{R}^{mn}\times \mathbb{R}^{m+n} \times \mathbb{R}^{mn},
\end{equation}
where $ \sigma >0 $ is a given constant. Therefore, a natural idea is to apply the semismooth Newton (SSN) method \cite{qi1993nonsmooth}, i.e., a nonsmooth version of Newton method, for solving \eqref{eq-kkt-mapping}. In particular, at each iteration of the SSN method, the following linear system is solved
\begin{equation}
	\label{eq:ssn-system}
	\begin{pmatrix}
		A & 0 & 0 \\
		0 & -A^T & -I_{mn} \\
		I_{mn} - V & 0 & \sigma V
	\end{pmatrix}
	\begin{pmatrix}
		\Delta x \\ \Delta y \\ \Delta z
	\end{pmatrix} = \begin{pmatrix}
		r_p:= d - Ax \\ r_d:=A^Ty + z - c \\ r_c:= \Pi_+(x - \sigma z) - x
	\end{pmatrix},
\end{equation}
where $(x, y, z)$ denotes the current iterate and $V = \mathrm{Diag}(v) \in \mathbb{R}^{mn\times mn}$ is an element in the Clarke's generalized Jacobian (see \cite{clarke1990optimization}) of $\Pi_+(x - \sigma z)$. Specifically, for given $(x, z)$, we may choose 
\[
v_{i} = \left\{
\begin{array}{ll}
	 1, & \textrm{if }\; (x -\sigma z)_i \geq 0, \\
	%\in [0,1], & \textrm{if }\; (x -\sigma z)_i = 0, \\
	0, & \textrm{otherwise},
\end{array}
\right.  \quad i= 1,\dots, mn. 
\]
By the expression of $V$, we see that when $(x, z)$ are nearly optimal, $V$ could have a large number of zeros diagonal elements. Thus, the sparsity structure of the optimal solution can be utilized in the SSN method, which is an attractive feature when compared with the IPM. However, there are two issues that  prevent us from applying the SSN method directly:
\begin{enumerate}
	\item The mapping $F$ is nonsmooth due to the nonsmoothness of $\Pi_+$. Hence, it is difficult to find a valid merit function for which a line search scheme {for the SSN} is well-defined. As a consequence, the globalization of the SSN method is a challenging difficulty that has yet to be resolved. 
	\item It is not clear how to reduce the large-scale linear system \eqref{eq:ssn-system}, which is typically singular,  to a certain normal linear system to save computational cost as in the case of the interior point method.
\end{enumerate}

%%%%%%%%%%%%%%%%%%%%%%%
\section{A Squared Smoothing Newton Method} \label{sec:sqsn}

Note that the KKT equation \eqref{eq-kkt-mapping} can be reduced to {finding the root of the following mapping $E(x,y)$:}
\begin{equation}\label{eq-new-kkt-mapping}
	0 \;=\; E(x,y):= \begin{pmatrix}
		Ax-d \\ 
		x - \Pi_+(x + \sigma (A^Ty - c))
	\end{pmatrix},\quad \;(x, y) \in \mathbb{R}^{mn}\times \mathbb{R}^{m+n},
\end{equation}
by eliminating the variable $ z = c - A^Ty $. In this section, we focus on presenting our main algorithm, the squared smoothing Newton method via the Huber smoothing function, for solving the nonsmooth system \eqref{eq-new-kkt-mapping}. Existing smoothing Newton methods have been developed and studied extensively and were mainly applied for solving conic programming problems, complementarity problems, and variational inequalities. We refer the readers to \cite{ chen1998global, kong2008regularized, qi2000new, sun2004squared} for more details. 

\subsection{The Huber smoothing function}\label{subsec:huber}

The key idea of a smoothing method  \cite{sun2001solving} is to construct a smoothing approximation function $ \mathcal{E}:\mathbb{R}_{++}\times \mathbb{R}^{mn}\times \mathbb{R}^{m+n} \rightarrow  \mathbb{R}^{mn}\times \mathbb{R}^{m+n}$ for the mapping $ E$, and then apply the Newton method for solving the smoothed system $ \mathcal{E}(\epsilon, x, y) = 0$, for $ (\epsilon, x, y)\in \mathbb{R}\times \mathbb{R}^{mn} \times \mathbb{R}^{m+n} $. To this end, it is essential to construct a smoothing approximation of the projection operator $ \Pi_+(\cdot) $, namely, $ \Phi:\mathbb{R}_{++}\times \mathbb{R}^{mn} \rightarrow \mathbb{R}^{mn} $, such that for any $ \epsilon > 0 $, $ \Phi(\epsilon,\cdot) $ is continuously differentiable on $ \mathbb{R}^{mn} $ and for any $ x\in \mathbb{R}^{mn} $, it holds that $ \left\lVert \Phi(\epsilon, x) - \Pi_+(x) \right\rVert \rightarrow 0 $ as $ \epsilon \downarrow 0 $. 

Note that a smoothing approximation function of $ \Pi_+(\cdot) $ strongly depends on the smoothing function for the plus function 
\[
	\pi(t) = \max\{0, t\} , \quad \forall\; t\in \mathbb{R}.
\]
To the best of our knowledge, the Chen-Harker-Kanzow-Smale (CHKS) smoothing function \cite{chen1993non, kanzow1996some, smale2000algorithms}, defined as
\[
	\xi(\epsilon, t):= \frac{1}{2}\left(\sqrt{t^2+4\epsilon^2} + t\right),\quad \forall \; (\epsilon, t)\in \mathbb{R}\times \mathbb{R},
\] 
is the most commonly used smoothing function for $ \pi(\cdot) $ in the literature. Readers are referred to \cite{qi2000new} for other smoothing functions, whose properties are also well studied. It is easy to see that $ \xi(\epsilon, t) $ maps any negative number $ t $ to a positive value when $ \epsilon\neq 0 $, while the function $ \pi(\cdot) $ maps any negative number $ t $ to zero. Hence, the CHKS smoothing function does not preserve the sparsity structure of $\pi(\cdot)$. As a result, a smoothing Newton method developed based on the CHKS smoothing function is not able to exploit the sparsity in the solutions of the OT problem to reduce the cost of solving the linear system at each  iteration of the algorithm.

To resolve the issue just mentioned, we propose to use the Huber smoothing function (a translation of the function considered in \cite{zang1980smoothing}) which also maps any negative number to zero. In particular, the Huber smoothing function we shall use in this paper is defined as:
\begin{equation}\label{eq-huber-function}
	h(\epsilon, t):= \left\{ 
		\begin{array}{ll}
			t - \frac{\lvert \epsilon \rvert}{2}, & t \geq  \lvert \epsilon \rvert, \\
			\frac{t^2}{2\lvert \epsilon \rvert}, & 0 < t < \lvert \epsilon \rvert, \\
			0, & t \leq 0,
		\end{array}
	\right. \quad \forall \;(\epsilon, t)\in \mathbb{R}\backslash \{0\} \times \mathbb{R},\quad h(0,t) = \pi(t),\quad \forall  t\in \mathbb{R}.
\end{equation}

Clearly, $ h(\epsilon, t) $ is continuously differentiable for $ \epsilon \neq 0 $ and $ t\in \mathbb{R} $. In fact, we see that for $ \epsilon\neq 0 $,
\[
	h_1'(\epsilon, t) := \frac{\partial h}{\partial \epsilon}(\epsilon,t) = \left\{ 
	\begin{array}{ll}
		- \frac{1}{2}\mathrm{sign}(\epsilon), & t \geq  \lvert \epsilon \rvert, \\
		-\frac{t^2}{2\epsilon^2}\mathrm{sign}(\epsilon), & 0 < t < \lvert \epsilon \rvert, \\
		0, & t \leq 0,
	\end{array}
	\right.\quad h_2'(\epsilon, t) := \frac{\partial h}{\partial t}(\epsilon,t) = \left\{ 
	\begin{array}{ll}
		1, & t \geq  \lvert \epsilon \rvert, \\
		\frac{t}{\lvert \epsilon \rvert}, & 0 < t < \lvert \epsilon \rvert, \\
		0, & t \leq 0,
	\end{array}
	\right.. %\quad \forall t\in \mathbb{R}.
\]
Moreover, it is easy to check that for any $ t\in \mathbb{R} $,
\begin{equation*}
	\begin{aligned}
		\partial_Bh(0, t) = &\; \left\{v \mid v = \lim_{ {\epsilon^k\rightarrow 0}, t^k\rightarrow t} h'(\epsilon^k, t^k),\; \textrm{if exists}\right\} = \left\{
			\begin{array}{ll}
				\left(\pm \frac{1}{2}, 1\right), & t > 0, \\[3pt]
				\left\{\left(\pm\frac{1}{2}\xi^2, \xi\right)\mid \xi\in[0,1] \right\}, & t = 0, \\[3pt]
				(0, 0), & t < 0, 
			\end{array}
		\right. 
	\end{aligned}
\end{equation*}
and consequently, 
\begin{equation} \label{eq-partial-huber}
		 \partial h(0, t) = \mathrm{conv}(\partial_B h(0, t))
		= 
	\left\{
			\begin{array}{ll}
				T_+, & t > 0 \\
				T_0, & t = 0, \\
				\{(0, 0)\}, & t < 0, 
			\end{array}
		\right. 
\end{equation}
where 
\[
	\begin{aligned}
		T_+:= \;  \left\{(\nu, 1)\mid \nu \in \left[-\frac{1}{2}, \frac{1}{2}\right]\right\}, \quad
		{T_0 := \; \left\{\left(\nu, \xi \right)\mid \xi\in \left[0,1\right], \nu\in 
		\left[-\frac{1}{2}\xi,\frac{1}{2}\xi\right] \right\} .}
	\end{aligned}
\]
Recall that the generalized Jacobian of the plus function $ \pi $ has the following form
\[
	\partial \pi(t) =
	\left\{
		\begin{array}{ll}
			\{1\}, & t > 0, \\
			\left[0,1\right], & t = 0, \\
			\{0\}, & t < 0,
		\end{array}
	\right.\quad \forall t \in \mathbb{R}.
\]
Therefore, we can see that for any $ v\in \partial h(0,t) $, there exists a $u\in \partial \pi(t)$ such that $ v(0, h) \equiv u(h) $ for any $h\in \mathbb{R}^{mn}$. {One can also verify that the converse statement is true.}

\subsection{The main algorithm} \label{subsec:mainalg}
Using the Huber function $ h(\cdot,\cdot) $, we can construct the smoothing approximation for the projection operator $ \Pi_+(\cdot) $ as
\begin{equation}\label{eq-huber-Pi}
	\Phi(\epsilon, x) := (h(\epsilon, x_1),\dots, h(\epsilon, x_{mn}))^T,\quad \forall x\in \mathbb{R}^{mn},\; \epsilon \neq 0. 
\end{equation}
Given the above smoothing function $ \Phi(\cdot) $ for $ \Pi_+(\cdot) $, a smoothing approximation of $ E(\cdot,\cdot) $ can be constructed accordingly as follows:
\begin{equation}\label{eq-smoothing-E}
	\mathcal{E}(\epsilon,x,y):= 
	\begin{pmatrix}
		Ax + \kappa_p \epsilon y - d \\ (1+ \kappa_c \epsilon) x - \Phi(\epsilon, x + \sigma (A^Ty - c)) 
	\end{pmatrix},\quad \forall (\epsilon, x, y )\in \mathbb{R}_{++} \times \mathbb{R}^{mn} \times \mathbb{R}^{m+n},
\end{equation}
where $ \kappa_p > 0 $ and $ \kappa_c > 0 $ are two given parameters. Note that adding the perturbation terms $ \kappa_p\epsilon y $ and $ \kappa_c\epsilon x $ is essential in our algorithmic design, since these perturbations will ensure that the Jacobian of $ \mathcal{E} $ at $ (\epsilon,x,y) $ is nonsingular when $ \epsilon > 0 $. 

In the current literature, smoothing Newton methods can be roughly divided into two groups, the Jacobian smoothing Newton methods and the squared smoothing Newton methods (see \cite{qi1999survey} for a comprehensive survey). The global convergence of a Jacobian smoothing Newton method strongly relies on the so-called Jacobian consistency property and the existence of a positive constant $ \kappa  $ such that 
\begin{equation}\label{eq-jacobian-bdd}
	\left\lVert \mathcal{E}(\epsilon, x, y) - E(x,y)\right\rVert \leq \kappa \epsilon,\quad \forall \epsilon > 0, \; (x, y)\in \mathbb{R}^{mn}\times \mathbb{R}^{m+n}.
\end{equation}
Unfortunately, the property \eqref{eq-jacobian-bdd} does not hold for \eqref{eq-smoothing-E}. Therefore, in our setting, the Jacobian smoothing Newton method {is not} applicable. On the other hand, the squared smoothing Newton method tries to find a solution of $ \mathcal{E}(\epsilon, x, y)=0 $  by solving the following system 
\begin{equation}\label{eq-E-hat}
	\widehat{\mathcal{E}}(\epsilon, x, y) := 
	\begin{pmatrix}
		\epsilon \\ \mathcal{E}(\epsilon, x, y)
	\end{pmatrix} = 0,\quad (\epsilon, x, y )\in \mathbb{R} \times \mathbb{R}^{mn} \times \mathbb{R}^{m+n},
\end{equation}
via the Newton method. As we shall see shortly, the global convergence of the squared smoothing Newton (SqSN) method does not rely on strong conditions such as \eqref{eq-jacobian-bdd}. This explains why we choose the SqSN method in the present paper.

Associated with the system \eqref{eq-E-hat}, we can define the merit function $ \phi: \mathbb{R} \times \mathbb{R}^{mn} \times \mathbb{R}^{m+n} \rightarrow \mathbb{R}$:
\[
	\phi(\epsilon, x, y):= \left\lVert \widehat{\mathcal{E}}(\epsilon, x, y) \right\rVert^2,\quad  (\epsilon, x, y )\in \mathbb{R} \times \mathbb{R}^{mn} \times \mathbb{R}^{m+n},
\]
to ensure that the line search procedure within the SqSN method is well-defined.  We also define an auxiliary function $ \zeta: \mathbb{R} \times \mathbb{R}^{mn} \times \mathbb{R}^{m+n} \rightarrow \mathbb{R} $:
\[
	\zeta(\epsilon, x, y) := r\min\Big\{1, \left\lVert \widehat{\mathcal{E}}(\epsilon, x, y)\right\rVert^{1+\tau}\Big\},\quad \forall (\epsilon, x, y )\in \mathbb{R} \times \mathbb{R}^{mn} \times \mathbb{R}^{m+n},
\]
which controls how the smoothing parameter $ \epsilon $ is driven to zero, where $ r\in (0,1) $ and $ \tau\in (0,1] $ are two given constants. Using these notations, we can now describe our squared smoothing Newton method via the Huber function in Algorithm \ref{alg:smoothingNewton}.

\begin{algorithm}[H]
\caption{A squared smoothing Newton (SqSN) method via the Huber function}\label{alg:smoothingNewton}
\begin{algorithmic}[1]
\Require Initial point $ (x^0, y^0)\in \mathbb{R}^{mn}\times \mathbb{R}^{m+n} $, $ \epsilon^0 > 0 $,  $ r\in (0,1) $ such that $ \delta:=r\epsilon^0 < 1$, $ \tau\in (0,1] $, $ \rho\in (0,1) $, and $ \mu\in (0,1/2) $.
\For{$ k\geq 0 $}
	\If{$ \widehat{\mathcal{E}}(\epsilon^k, x^k, y^k) = 0 $}
		\State Output: $ (\epsilon^k, x^k, y^k) $;	
	\Else
	\State Compute $ \zeta_k = \zeta(\epsilon^k, x^k, y^k) $;
	\State Find $ \Delta^k := (\Delta \epsilon^k; \Delta x^k; \Delta y^k) $ via solving the following linear system of equations  
	\begin{equation}\label{eq-alg-smoothingNewton-1}
		\widehat{\mathcal{E}}(\epsilon^k, x^k, y^k) + \widehat{\mathcal{E}}'(\epsilon^k, x^k, y^k)\Delta^k = (\zeta_k\epsilon^0; 0; 0);
	\end{equation}
	\State Compute $ \ell_k $ as the smallest nonnegative integer $ \ell $ satisfying
	\[
		\phi(\epsilon^k + \rho^\ell \Delta \epsilon^k, x^k + \rho^\ell \Delta x^k, y^k + \rho^\ell \Delta y^k) \leq \left[1 - 2\mu(1-\delta)\rho^\ell\right] \phi(\epsilon^k, x^k, y^k);
	\]
	\State Update $ (\epsilon^{k+1}, x^{k+1}, y^{k+1}) = (\epsilon^k + \rho^{\ell_k} \Delta \epsilon^k, x^k + \rho^{\ell_k} \Delta x^k, y^k + \rho^{\ell_k} \Delta y^k) $;
	\EndIf 
\EndFor 
\end{algorithmic}
\end{algorithm}

Before establishing the convergence properties of Algorithm \ref{alg:smoothingNewton}, let us first see how to solve the linear system \eqref{eq-alg-smoothingNewton-1} and also explain the potential advantage of the SqSN method compared to the IPM and the SSN method. Suppose $ \epsilon^k > 0 $, then, we see that 
	\[
		\widehat{\mathcal{E}}'(\epsilon^k, x^{k}, y^{k}) = 
		\begin{pmatrix}
			1 & 0 & 0 \\
			\kappa_py^k & A & \kappa_p\epsilon I_{m+n} \\[3pt]
			\kappa_cx^k - V_1^k & (1+\kappa_c\epsilon^k)I_{mn} - V_2^k & -\sigma V_2^kA^T
		\end{pmatrix},
	\]
	where $ V_1^k:= \Phi_1'(\epsilon^k, w^k)$ and $ V_2^k:= \Phi_2'(\epsilon^k, w^k) $ denote the partial derivatives of $ \Phi $ with respect to the first and second arguments at the referenced point {with $w^k:=x^k - \sigma(c-A^Ty^k)$,} respectively. Note also that $ \left\lvert(V_1^k)_{ii}\right \rvert \in [0, 1/2] $ and $ \left\lvert(V_2^k)_{ii}\right \rvert \in [0, 1] $ for any $ i = 1,\dots, mn $. From \eqref{eq-alg-smoothingNewton-1}, we get $ \Delta \epsilon^k = -\epsilon^k + \zeta_k\epsilon^0 $ and 
	\[
		\begin{pmatrix}
			A & \kappa_p\epsilon^kI_{m+n} \\[3pt]
			(1+\kappa_c\epsilon^k)I_{mn} - V_2^k & -\sigma V_2^kA^T
		\end{pmatrix}
		\begin{pmatrix}
			\Delta x^k \\
			\Delta y^k
		\end{pmatrix} = 
		\begin{pmatrix}
			r_p^k \\ r_c^k
		\end{pmatrix},
	\]
	where
	\[
		\begin{aligned}
			r_p^k:= &\; d - Ax^k - \kappa_p\epsilon^ky^k - \kappa_py^k\Delta \epsilon^k, \\
			r_c^k:= &\; \Phi(\epsilon^k, x^k + \sigma(A^Ty^k - c)) - (1+\kappa_c\epsilon^k)x^k - (\kappa_c x^k - V_1^k)\Delta \epsilon^k.
		\end{aligned}
	\]
	It then follows that 
	$
		\Delta x^k = \left[(1+\kappa_c\epsilon^k)I_{mn} - V_2^k\right]^{-1}\left(r_c^k + \sigma V_2^kA^T\Delta y^k\right).
	$
	Furthermore, we have that 
	\begin{equation}\label{eq-smoothing-update-y}
	    \left[\kappa_p\epsilon^k I_{m+n} + \sigma A\left((1+\kappa_c\epsilon^k)I_{mn} - V_2^k\right)^{-1}V_2^kA^T\right] \Delta y^k = r_p^k - A\left((1+\kappa_c\epsilon^k)I_{mn} - V_2^k\right)^{-1}r_c^k.
	\end{equation}
	Denote $ v^k:= \mathrm{diag}\left[ \left((1+\kappa_c\epsilon^k)I_{mn} - V_2^k\right)^{-1}V_2^k\right] \in \mathbb{R}^{mn} $ and recall that 
	$ w^k := x^k - \sigma (c-A^Ty^k ) $. It is clear that 
	\[
		v_i^k  \left\{
			\begin{array}{ll}
				>0 & \text{if } w_i^k > 0 \\
				=0 & \text{if } w_i^k \leq 0
			\end{array}
		\right., \quad \forall i = 1,\dots, mn.
	\]
	Keeping in mind that $ z^k = c - A^Ty^k $ for $ k\geq 0 $, we see that when $ (x^k, y^k) $ is close to the optimal solution pair, $ w^k $ is expected to have a large number of nonpositive elements {and hence $v^k$ to be highly sparse}. Moreover, let $ V^k:= \mathrm{mat}(v^k) \in \mathbb{R}^{m\times n} $. By Fact \ref{fact-ADAT}, we see that 
\begin{equation}
     \label{eq-ot-ADAT}
    A\left((1+\kappa_c\epsilon^k)I_{mn} - V_2^k\right)^{-1}V_2^kA^T = 
    \begin{pmatrix}
        \mathrm{Diag}(V^ke_n) & V^k \\
        (V^k)^T & \mathrm{Diag}((V^k)^Te_m)
    \end{pmatrix} \in \mathbb{S}^{m+n}.
\end{equation}	
When $ V^k $ is sparse, it is clear that the coefficient matrix for computing $ \Delta y^k $ is also sparse. This explains the appealing feature of Algorithm \ref{alg:smoothingNewton}	as  compared to the IPM, since the attractive sparsity structure of the optimal solution is exploited in Algorithm \ref{alg:smoothingNewton}. Moreover, we are always able to reduce \eqref{eq-alg-smoothingNewton-1} to the smaller-scale normal system\eqref{eq-smoothing-update-y} in Algorithm \ref{alg:smoothingNewton}, which is not the case in the SSN method. We can also see that the matrix $ A $ is not explicitly required in the implementation of Algorithm \ref{alg:smoothingNewton}, which is another advantage of the proposed method compared to an IPM. Lastly, let $J:=\diag(I_m, -I_n)$, then one can see that 
    \begin{equation}\label{eq-ot-JADAJ}
        JA\left((1+\kappa_c\epsilon^k)I_{mn} - V_2^k\right)^{-1}V_2^kA^TJ = 
        \begin{pmatrix}
            \mathrm{Diag}(V^ke_n) & -V^k \\
            -(V^k)^T & \mathrm{Diag}((V^k)^Te_m)
        \end{pmatrix},
    \end{equation}
which corresponds to the Laplacian matrix of a sparse bipartite graph defined by $V^k$ \footnote{The bipartite graph has $m$ sources and $n$ sinks. There is an edge from source $i$ to sink $j$ if and only if $V^k_{ij} > 0$ for $1\leq i \leq m$ and $1\leq j\leq n$.}. {First, one can identify all the connected components of the bipartite graph that are disjoint from each other. The matrix entries linking different components are always zero. Next, by permuting the coefficient matrix of the Newton equation, it can be converted into a block diagonal matrix. Finally, the Newton equation can be decomposed into several smaller equations. Each coefficient matrix of the smaller equation has a similar pattern as \eqref{eq-ot-JADAJ}. } In this way, the computational effort for computing the search direction may be further reduced.

%%%%%%%%%%%%%%%%%%%%%%%%
\subsection{Convergence properties of Algorithm \ref{alg:smoothingNewton}}\label{subsec:convergence}

For the rest of this section, we shall establish the convergence properties of Algorithm \ref{alg:smoothingNewton} which allow us to resolve the two difficulties of the SSN method as mentioned at the end of the last section. The tools for our analysis are adapted from those in the literature (see, e.g., \cite{sun2004squared}). 

We first show that the coefficient matrix in \eqref{eq-alg-smoothingNewton-1} is nonsingular whenever $ \epsilon^k > 0 $.
\begin{lemma}
\label{lemma-nonsingular}
$ \widehat{\mathcal{E}}'(\epsilon, x, y) $ is nonsingular for any $ (\epsilon, x, y) \in \mathbb{R}_{++}\times \mathbb{R}^{mn}\times \mathbb{R}^{m+n}  $ . 
\end{lemma}
\begin{proof}
	For any $ \epsilon > 0 $, it is easy to see that $ \widehat{\mathcal{E}} $ is continuously differentiable, and it holds that
	\begin{equation*}
		\widehat{\mathcal{E}}'(\epsilon, x, y) = 
		\begin{pmatrix}
			1 & 0 & 0 \\
			\kappa_py & A & \kappa_p\epsilon I_{m+n} \\
			\kappa_cx - V_1 & (1+\kappa_c\epsilon)I_{mn} - V_2 & -\sigma V_2A^T
		\end{pmatrix},
	\end{equation*}
	where $ V_1:= \Phi_1'(\epsilon, w)$ and $ V_2:= \Phi_2'(\epsilon, w) $ denote the partial derivatives of $ \Phi $ with respect to the first and the second arguments at the referenced point with $w:=x - \sigma(c-A^Ty)$, respectively. To show that $ \widehat{\mathcal{E}}'(\epsilon, x, y) $ is nonsingular, it suffices to show that the following linear system
	\[
		\widehat{\mathcal{E}}'(\epsilon, x, y)(\Delta \epsilon; \Delta x; \Delta y) = 0
	\]
	only has a trivial solution. It is obvious that $ \Delta \epsilon = 0 $. Since 
	\[
		\left[(1+\kappa_c\epsilon)I_{mn} - V_2\right]\Delta x -\sigma V_2A^T\Delta y = 0,
	\]
	it follows that $\Delta x = \sigma \left[(1+\kappa_c\epsilon)I_{mn} -V_2 \right]^{-1}V_2A^T\Delta y$. As a consequence, we get 
	\[
		\left[\kappa_p\epsilon I_{m+n} + \sigma A \left((1+\kappa_c\epsilon)I_{mn} -V_2 \right)^{-1}V_2 A^T \right]\Delta y = 0.
	\]
	Clearly, the coefficient matrix of the above equation is symmetric positive definite. It then follows that $ \Delta y = 0 $ which further implies that $ \Delta x = 0 $. Therefore, the proof is completed.
\end{proof}

Lemma \ref{lemma-nonsingular} shows that the linear system in \eqref{eq-alg-smoothingNewton-1} is always solvable, and the Newton direction $ (\Delta^k, \Delta x^k, \Delta y^k) $ is well-defined. Next, we shall show that the line search procedure in Algorithm \ref{alg:smoothingNewton} is well-defined, i.e., $ \ell_k $ is always finite as long as $ \epsilon^k $ is positive. 

\begin{lemma}
\label{lemma-line-search}
{For any $ (\epsilon, x, y) $ with $ \epsilon > 0 $, there exist a positive scalar $ \bar \alpha \in (0, 1] $ such that for any $ \alpha\in (0, \bar \alpha] $, it holds that}
\[
	\phi(\epsilon+\alpha\Delta \epsilon, x + \alpha \Delta x, y + \alpha \Delta y) \leq \left[1 - 2\mu(1-\delta)\alpha \right]\phi(\epsilon, x, y),
\]
where $ (\Delta\epsilon, \Delta x, \Delta y) $ satisfies
\[
	\widehat{\mathcal{E}}(\epsilon, x, y) + \widehat{\mathcal{E}}'(\epsilon, x, y)(\Delta \epsilon; \Delta x; \Delta y) = (\zeta(\epsilon, x, y)\epsilon^0; 0; 0),
\]
$ \mu\in(0,1/2) $, $ \epsilon^0 > 0 $, and $ \delta := r\epsilon^0 < 1 $.
\end{lemma}
\begin{proof}
	First, since $ \epsilon > 0 $, by Lemma \ref{lemma-nonsingular}, we see that $ \widehat{\mathcal{E}}'(\epsilon, x, y) $ is nonsingular. Hence, $ (\Delta\epsilon, \Delta x, \Delta y) $ is well-defined. Next, we can check that
	\begin{equation}\label{eq-lemma-line-search-1}
		\begin{aligned}
			& \left\langle \nabla \phi(\epsilon, x, y), (\Delta \epsilon, \Delta x, \Delta y)\right\rangle 
			\;= \; \left\langle 2\nabla \widehat{\mathcal{E}}(\epsilon, x, y) \widehat{\mathcal{E}}(\epsilon, x, y), (\Delta \epsilon; \Delta x; \Delta y)\right\rangle \\
			= &\; \left\langle 2\widehat{\mathcal{E}}(\epsilon, x, y), \widehat{\mathcal{E}}'(\epsilon, x, y)(\Delta \epsilon; \Delta x; \Delta y)\right\rangle 
			\;=\; \left\langle 2\widehat{\mathcal{E}}(\epsilon, x, y), (\zeta(\epsilon, x, y)\epsilon^0; 0; 0) - \widehat{\mathcal{E}}(\epsilon, x, y)\right\rangle \\
			\leq &\; -2\phi(\epsilon, x, y) + 2r\epsilon^0\left\lVert\widehat{\mathcal{E}}(x,y,z)\right\rVert \min\left\{1, \left\lVert \widehat{\mathcal{E}}(\epsilon, x, y)\right\rVert^{1+\tau}\right\}.
		\end{aligned}
	\end{equation}
We consider two possible cases: $ \left\lVert \widehat{\mathcal{E}}(\epsilon, x, y)\right\rVert > 1 $ and $ \left\lVert \widehat{\mathcal{E}}(\epsilon, x, y)\right\rVert \leq 1 $. If $ \left\lVert \widehat{\mathcal{E}}(\epsilon, x, y)\right\rVert > 1 $, then \eqref{eq-lemma-line-search-1} implies that 
\[
	\left\langle \nabla \phi(\epsilon, x, y), (\Delta \epsilon, \Delta x, \Delta y)\right\rangle \leq -2\phi(\epsilon, x, y) + 2r\epsilon^0\left\lVert\widehat{\mathcal{E}}(x,y,z)\right\rVert  \leq -2(1- \delta)\phi(\epsilon, x, y).
\]
If $ \left\lVert \widehat{\mathcal{E}}(\epsilon, x, y)\right\rVert \leq 1 $, then  \eqref{eq-lemma-line-search-1} implies that 
\[
	\begin{aligned}
		\left\langle \nabla \phi(\epsilon, x, y), (\Delta \epsilon, \Delta x, \Delta y)\right\rangle 
		 \leq 
%		 &\; - 2\phi(\epsilon, x, y) + 2r\epsilon^0\left\lVert\widehat{\mathcal{E}}(x,y,z)\right\rVert\left\lVert \widehat{\mathcal{E}}(\epsilon, x, y)\right\rVert^{1+\tau} \\
%		  = 
		  &\;  - 2\phi(\epsilon, x, y) + 2r\epsilon^0\left\lVert \widehat{\mathcal{E}}(\epsilon, x, y)\right\rVert^{2+\tau}  \\
		  \leq &\; - 2\phi(\epsilon, x, y) + 2r\epsilon^0\phi(\epsilon, x, y) \\
		  = &\; -2(1- \delta)\phi(\epsilon, x, y).
	\end{aligned}
\]
Thus, in both cases, it always holds that
\[
\left\langle \nabla \phi(\epsilon, x, y), (\Delta \epsilon, \Delta x, \Delta y)\right\rangle\leq -2(1- \delta)\phi(\epsilon, x, y).
\]
Now by the Taylor expansion, we get  
\[
	\begin{aligned}
		\phi(\epsilon+\alpha\Delta \epsilon, x + \alpha \Delta x, y + \alpha \Delta y)
		= &\; \phi(\epsilon, x, y) + \alpha \left\langle \nabla \phi(\epsilon, x, y), (\Delta \epsilon, \Delta x, \Delta y)\right\rangle + o(\alpha) \\
		\leq  &\; \phi(\epsilon, x, y) - 2\alpha(1-\delta)\phi(\epsilon, x, y) + o(\alpha).
	\end{aligned}
\]
Using the fact that $ \mu \in (0, 1/2) $, there exists $ \bar\alpha \in (0,1] $ such that for $ \alpha\in (0, \bar\alpha] $, it always holds that
\[
	\phi(\epsilon+\alpha\Delta \epsilon, x + \alpha \Delta x, y + \alpha \Delta y) \leq \left[1 - 2\mu(1-\delta)\alpha \right]\phi(\epsilon, x, y),
\]
which completes the proof.
\end{proof}

Lemma \ref{lemma-line-search} indicates that when $ \epsilon^k >0 $, the line search procedure in Algorithm \ref{alg:smoothingNewton} is always well-defined. {Additionally, we need to show} that Algorithm \ref{alg:smoothingNewton} generates a sequence $ \{\epsilon^k\} $ that is always positive before termination. The next lemma shows that $ \epsilon^k $ is indeed lower bounded by $ \zeta_k\epsilon^0 $, which is positive before termination. 

\begin{lemma}
	\label{lemma:epsilon-lower-bounded}
	Suppose at the $ k $-th iteration of Algorithm \ref{alg:smoothingNewton} that $ \epsilon^k>0 $ and $ \epsilon^k \geq \zeta(\epsilon^k, x^k, y^k)\epsilon^0 $. Then, for any $ \alpha \in [0,1] $ satisfying
	\begin{equation} \label{eq-lemma-epsilon-1}
		\phi(\epsilon^k + \alpha \Delta \epsilon^k, x^k + \alpha \Delta x^k, y^k + \alpha \Delta y^k) \leq \left[1 - 2\mu(1-\delta)\alpha\right] \phi(\epsilon^k, x^k, y^k),
	\end{equation}
	it holds that
	$
		\epsilon^k + \alpha\Delta \epsilon^k \geq \zeta(\epsilon^k + \alpha \Delta \epsilon^k, x^k + \alpha \Delta x^k, y^k + \alpha \Delta y^k)\epsilon^0.
	$
\end{lemma}
\begin{proof}
	From \eqref{eq-alg-smoothingNewton-1}, it is not difficult to see that $ \Delta \epsilon^k = -\epsilon^k + \zeta(\epsilon^k, x^k, y^k)\epsilon^0 $. Hence, $ \Delta \epsilon^k \leq 0 $, and for any $ \alpha \in [0,1] $, it holds that 
	\[
		\begin{aligned}
			&\; \epsilon^k + \alpha\Delta \epsilon^k - \zeta(\epsilon^k + \alpha \Delta \epsilon^k, x^k + \alpha \Delta x^k, y^k + \alpha \Delta y^k)\epsilon^0 \\
			\geq &\; \epsilon^k + \Delta \epsilon^k - \zeta(\epsilon^k + \alpha \Delta \epsilon^k, x^k + \alpha \Delta x^k, y^k + \alpha \Delta y^k)\epsilon^0 \\
			= &\; \zeta(\epsilon^k, x^k, y^k)\epsilon^0 - \zeta(\epsilon^k + \alpha \Delta \epsilon^k, x^k + \alpha \Delta x^k, y^k + \alpha \Delta y^k)\epsilon^0 \\
			\geq &\; 0,
		\end{aligned}
	\] 
	where in the last inequality we have used the fact that \eqref{eq-lemma-epsilon-1} implies that 
	\[
		\zeta(\epsilon^k, x^k, y^k) \geq \zeta(\epsilon^k + \alpha \Delta \epsilon^k, x^k + \alpha \Delta x^k, y^k + \alpha \Delta y^k)
	\]
	by the definition of the function $ \zeta $. Therefore, the proof is completed. 
\end{proof}
The above lemma also explains the role of the auxiliary function $ \zeta $ in the algorithmic design. In particular, when we choose $ r $ to be small or $ \tau  $ to be close to one, then the lower bound $ \zeta_k\epsilon^0 $ is small. As a result, $ \epsilon^k $ will also be small. However, a smaller $ \epsilon^k $ will worsen the conditioning of the linear system \eqref{eq-alg-smoothingNewton-1}. On the other hand, if we choose $ r  $ to be large or $ \tau $ to be close to zero, then, $ \epsilon^k  $ will be further away from zero and the algorithm would need more iterations to converge to an optimal solution. Therefore, there is a trade-off in choosing the parameters $ r $ and $ \tau  $ in  the practical implementation of Algorithm \ref{alg:smoothingNewton}. 

Now, we are ready to state the first convergence result of the Algorithm \ref{alg:smoothingNewton}. 
\begin{theorem}
\label{thm-global-convergence}
Algorithm \ref{alg:smoothingNewton} is well-defined and generates an infinite sequence $ \{(\epsilon^k, x^k, y^k)\}   $	such that $ \epsilon^k \geq \zeta(\epsilon^k, x^k, y^k)\epsilon^0 $ with the property that any accumulation point $ (\bar\epsilon, \bar x, \bar y) $ of the sequence $ \{(\epsilon^k, x^k, y^k)\}   $ is a solution of the nonlinear system $ \widehat{\mathcal{E}}(\epsilon, x, y) = 0  $.
\end{theorem}
\begin{proof}
	It follows from Lemma \ref{lemma-nonsingular}, \ref{lemma-line-search} and  \ref{lemma:epsilon-lower-bounded}  that Algorithm \ref{alg:smoothingNewton} is well-defined, and it generates an infinite sequence $ \{(\epsilon^k, x^k, y^k)\}   $	such that $ \epsilon^k \geq \zeta(\epsilon^k, x^k, y^k)\epsilon^0 $. Let $ (\bar\epsilon, \bar x, \bar y) $ be any accumulation point of the sequence $ \{(\epsilon^k, x^k, y^k)\}   $, if it exists. By taking a subsequence if necessary, we may assume without loss of generality that  $ \{(\epsilon^k, x^k, y^k)\}   $ converges to $ (\bar\epsilon, \bar x, \bar y) $. Since the line-search scheme is well-defined, it follows that 
	\[
		\phi(\epsilon^{k+1}, x^{k+1}, y^{k+1}) < \phi(\epsilon^k, x^k, y^k),\quad \forall k\geq 0.
	\] 
	That is, the sequence $ \{\phi(\epsilon^k, x^k, y^k) \} $ is monotonically decreasing. By the definition of the function $ \zeta $, we see that the sequence $ \{\zeta_k\} $ is also monotonically decreasing. Hence, there exist $ \bar \phi  $ and $ \bar \zeta  $ such that
	\[
		\phi(\epsilon^k, x^k, y^k)\rightarrow \bar \phi, \quad \zeta_k \rightarrow \bar \zeta,\quad\text{as}\; k\rightarrow \infty.
	\]
	As a consequence of the continuity of the function $ \phi $, we see that 
	\[
		\phi(\bar\epsilon,\bar x, \bar y) = \bar \phi \geq 0, \quad \zeta(\bar \epsilon, \bar x, \bar y) = \bar \zeta \geq 0, \quad \bar \epsilon \geq \zeta(\bar \epsilon, \bar x, \bar y)\epsilon^0. 
	\]
	
	We next show that $ \bar\phi = \phi(\bar \epsilon, \bar x, \bar y) = 0 $ by contradiction. To this end, let us assume that $ \bar\phi > 0 $. This implies that $ \bar \zeta > 0 $ and that $ \bar \epsilon \geq \bar\zeta\epsilon^0 > 0 $. Then, by Lemma \ref{lemma-line-search}, we see that there exist an open neighbourhood $ \mathcal{U} $ of $ (\bar \epsilon, \bar x, \bar y) $ and $ \bar \alpha \in (0,1]  $ such that $ (\epsilon^k, x^k, y^k) \in \mathcal{U} $ with $ \epsilon^k > 0 $ and for any $ \alpha \in (0, \bar \alpha] $, it holds that  for $ k\geq 0 $ sufficiently large,
	 \[
	 	\phi(\epsilon^k+\alpha \Delta\epsilon^k, x^k + \alpha \Delta x^k, y^k + \alpha \Delta y^k)\leq \left[1 - 2\mu(1-\delta)\alpha\right]\phi(\epsilon^k, x^k, y^k).
	 \]
	The existence of the fixed number $ \bar\alpha \in (0,1] $ further indicates that there exists a fixed nonnegative integer $ \ell  $ such that $ \rho^\ell \in (0, \bar \alpha] $ and $ \rho^{\ell_k}\geq \rho^\ell $ for all $ k\geq 0 $ sufficiently large. Therefore, we can check that 
	 \[
	 	\phi(\epsilon^{k+1}, x^{k+1}, y^{k+1}) \leq  \left[1 - 2\mu(1-\delta)\rho^{\ell_k}\right]\phi(\epsilon^k, x^k, y^k) \leq \left[1 - 2\mu(1-\delta)\rho^{\ell}\right]\phi(\epsilon^k, x^k, y^k),
	 \]
	for all $ k\geq 0 $ sufficiently large. Then, by letting $ k\rightarrow \infty $, we see that $\bar\phi \leq \left[1 - 2\mu(1-\delta)\rho^{\ell}\right] \bar\phi$. This implies that $ \bar\phi \leq 0 $, which contradicts the assumption that $ \bar\phi >0 $. Therefore, we conclude that $ \bar \phi = 0 $, i.e., $ \widehat{\mathcal{E}}(\bar \epsilon, \bar x, \bar y) = 0 $. This completes the proof.
\end{proof}

From our previous convergence analysis, we see how the natural merit function $ \phi $ helps us to establish the convergence properties. Without such a continuously differentiable merit function, further stronger assumptions and much more complicated analysis may be required. Therefore, we are able to design a Newton-type method that is able to exploit the solution sparsity and admit global convergence properties by employing the SqSN method via the Huber function. Following the classical results on the local convergence rate of Newton-type methods, next we further show that Algorithm \ref{alg:smoothingNewton} has a local superlinear convergence rate under the assumption that every element of $ \partial\widehat{\mathcal{E}}(\bar \epsilon, \bar x, \bar y) $ is nonsingular. 

\begin{theorem}
	\label{thm-superlinear-rate}
	Let $ (\bar \epsilon,\bar x, \bar y) $ be any accumulation point  of the sequence $ \{(\epsilon^k, x^k, y^k)\} $ (if it exists) generated by the Algorithm \ref{alg:smoothingNewton}. Suppose that every element of $ \partial\widehat{\mathcal{E}}(\bar \epsilon, \bar x, \bar y) $ is nonsingular, then the whole sequence $ \{(\epsilon^k, x^k, y^k)\} $ converges to $ (\bar \epsilon,\bar x, \bar y) $ superlinearly. In particular, it holds that 
	\[
		\left\lVert \begin{pmatrix}
			\epsilon^{k+1} - \bar \epsilon \\
			x^{k+1} - \bar x \\
			y^{k+1} - \bar y
		\end{pmatrix} \right\rVert = O\left(\left\lVert \begin{pmatrix}
			\epsilon^{k} - \bar \epsilon \\
			x^{k} - \bar x \\
			y^{k} - \bar y
		\end{pmatrix} \right\rVert^{1+\tau}\right).
	\]
\end{theorem}
\begin{proof} The proof is given in Appendix A.
\end{proof}

It turns out that the conditions ensuring the local convergence rate of the Algorithm \ref{alg:smoothingNewton} are highly related to those for the SSN method. Classical results (see e.g., \cite[Theorem 3.2]{qi1993nonsmooth}) show that the SSN method admits local quadratic convergence rate under the conditions that $ F $ is strongly semismooth and every element of $ \partial F(\bar x, \bar y, \bar z) $ is nonsingular, where $ (\bar x, \bar y, \bar z) $ is a solution to the KKT system $ F(x, y, z) = 0 $ in \eqref{eq-kkt-mapping}. The following lemma shows the connection between the nonsingularity of $ \partial F(\bar x, \bar y, \bar z) $ and the nonsingularity of $ \partial \widehat{\mathcal{E}}(0, \bar x, \bar y) $. 

\begin{lemma}
	\label{lemma-kkt-jacobian-equivalent}
	Let $ (\bar x, \bar y, \bar z) $ be such that $ F(\bar x, \bar y, \bar z) = 0 $. Then the following two statements are equivalent.
	\begin{enumerate}
		\item Every element of $ \partial F(\bar x, \bar y, \bar z) $ is nonsingular.
		\item Every element of $ \partial \widehat{\mathcal{E}}(0, \bar x, \bar y) $ is nonsingular.
	\end{enumerate}
\end{lemma}
\begin{proof}
	We first show statement 1 implies statement 2. Suppose that $ U \in \partial\widehat{\mathcal{E}}(0, \bar x, \bar y) $. Then there exists 
	$ V\in {\partial \Phi(0, \bar x - \sigma (c-A^T\bar y))} $ such that 
	\[
		U(\Delta \epsilon, \Delta x, \Delta y) = 
		\begin{pmatrix}
			\Delta \epsilon \\
			A\Delta x  \\
			\Delta x - V(\Delta \epsilon, \Delta x + \sigma A^T\Delta y)
		\end{pmatrix},\quad \forall (\Delta\epsilon, \Delta x, \Delta y) \in \mathbb{R}\times \mathbb{R}^{mn}\times \mathbb{R}^{m+n}.
	\]
	To show that $ U $ is nonsingular, it suffices to show that $ U(\Delta\epsilon,\Delta x, \Delta y) = 0 $, implies that $ (\Delta \epsilon, \Delta x, \Delta y) = 0 $. To this end, we assume that $ U(\Delta\epsilon,\Delta x, \Delta y) = 0 $. Then clearly, it holds that $ \Delta \epsilon = 0 $ and 
	\begin{equation} \label{eq-lemma-equivalent-1}
		\begin{pmatrix}
			A \Delta x \\
			-A^T \Delta y - \Delta z \\
			\Delta x - V(0, \Delta x - \sigma \Delta z)
		\end{pmatrix} = 
		\begin{pmatrix}
			0 \\ 0 \\ 0
		\end{pmatrix},
	\end{equation}
	by denoting that {$ \Delta z:= -A^T \Delta y $.}
	 By the expression of $ \partial h(0, t) $ given in \eqref{eq-partial-huber}, we see that there exists $ W\in \partial \Pi_+(\bar x - \sigma(c-A^T\bar y)) $ such that $ V(0, h) = W(h) $ for any $ h\in \mathbb{R}^{mn} $. Then the last equation of linear system  \eqref{eq-lemma-equivalent-1} can be rewritten as $\Delta x - W(\Delta x - \sigma \Delta z)=0$, and the resulting linear system only has the trivial solution since statement 1 holds true. Therefore, $ U $ is nonsingular. Similarly, one can show that statement 2 implies statement 1. The proof is completed.  
\end{proof}

By the above lemma, we can see that the conditions for ensuring the local convergence rate for Algorithm \ref{alg:smoothingNewton} are not stronger than those for the SSN method. {Furthermore, the nonsingularity assumption is equivalent to the primal and dual linear independence constraint qualification (LICQ) and the local Lipschitz homeomorphism of $F(x,y,z)$ near $(\bar x,\bar y,\bar z)$ (see e.g., \cite{kong2011equivalent,chan2008constraint}). These equivalent conditions are quite strong and may not hold generally. Surprisingly, in our numerical experiments, we indeed observe such local superlinear convergence empirically. This is reasonable because the nonsingularity assumption may hold in the selected experiments. Moreover, the nonsingularity assumption is sufficient but may not be necessary for the superlinear convergence. For example, the superlinear convergence of the projected SSN method is established in \cite{hu2021semismooth} under the local smoothness and the error bound conditions on a submanifold, without the nonsingularity assumption.}

\section{Extension to the Wasserstein Barycenter Problem}\label{sec:wbp}

In this section, we proceed to extend our Algorithm \ref{alg:smoothingNewton} to solve the Wasserstein barycenter problem.

\subsection{Problem Statement}

First, we briefly recall the Wasserstein distance and describe the problem of computing a Wasserstein barycenter for a set of discrete probability distributions with finite support points. Given two discrete distributions $\mathcal{P}^{(1)}=\{(a_i^{(1)},\,{q}_i^{(1)}): i = 1, \dots, m_1\}$ and $\mathcal{P}^{(2)}=\{(a_i^{(2)}, {q}_i^{(2)}): i = 1, \dots, m_2\}$, the $p$-Wasserstein distance $\mathcal{W}_p(\mathcal{P}^{(1)}, \mathcal{P}^{(2)})$ is defined by the following OT problem:
\begin{equation*}
\begin{array}{cl}
\left(\mathcal{W}_p\left(\mathcal{P}^{(1)}, \mathcal{P}^{(2)}\right)\right)^p:= &
 \min 
 \limits_{X \in \mathbb{R}^{m_1 \times m_2}} \Big\{
 \left\langle X, \mathcal{D}\left(\mathcal{P}^{(1)}, \mathcal{P}^{(2)}\right)\right\rangle 
 \mid 
 X^{\top} {{e}_{{m}_{{1}}}}={a}^{{(1)}},\; X {{e}_{{m}_{{2}}}}={a}^{{(2)}}, \; X \geq {0} \Big\},
\end{array}
\end{equation*}
where $\mathcal{D}(\mathcal{P}^{(1)}, \mathcal{P}^{(2)}) \in \mathbb{R}^{m_1 \times m_2}$ is the distance matrix with $\mathcal{D}(\mathcal{P}^{(1)}, \mathcal{P}^{(2)})_{i j}=\lVert q_{{i}}^{(1)}-q_j^{(2)}\rVert_p^p$ and $p\geq 1$. Then, given a set of discrete probability distributions $\{\mathcal{P}^{(t)}\}_{t=1}^N$ with $\mathcal{P}^{(t)}=\{(a_i^{(t)},\,{q}_i^{(t)}): 1\leq i \leq n\}$, a $p$-Wasserstein barycenter $\mathcal{P}:=\{(w_i,\,{q}_i): i = 1, \dots, m\}$ with $m$ support points is an optimal solution to the following problem:
\begin{equation}\label{eq-nonconvex-wbp}
\min  \Big\{ 
\sum_{t=1}^N \gamma_t\left(\mathcal{W}_p\left(\mathcal{P}, \mathcal{P}^{(t)}\right)\right)^p \mid  w\in \mathbb{R}^m_+,\; e_m^\top w = 1, \; q_1,\ldots,q_m \in \mathbb{R}^d\Big\}
\end{equation}
for given weights $\left(\gamma_1, \dots, \gamma_N\right)$ satisfying $\sum_{t=1}^N \gamma_t=1$ and $\gamma_t > 0,\; t=1, \dots, N$. {Note that} problem \eqref{eq-nonconvex-wbp} is a non-convex {optimization} problem, in which one needs to find the optimal support ${q}:=\{q_1,\dots,q_m\}$ and the optimal weight vector {${w}:=(w_1,\dots,w_m)$} of a barycenter simultaneously. In many real applications, the support ${q}$ of a barycenter is {pre-specified}. Thus, one only needs to find the weight vector ${w}$ of a barycenter. In view of this, from now on, we assume that the support ${q}$ is given. Consequently, the problem \eqref{eq-nonconvex-wbp} reduces to the WB problem with fixed support points:

\begin{equation}\label{eq-wbp}
\begin{aligned}
&\min\limits_{{w},\,\{\Pi^{(t)}\}}~{\textstyle\sum^N_{t=1}}\langle D^{(t)}, \,\Pi^{(t)} \rangle  \\
&\hspace{0.5cm}\mathrm{s.t.} \hspace{0.5cm} \Pi^{(t)}{e}_{m_t} = {w}, ~(\Pi^{(t)})^{\top}{e}_{m} = {a}^{(t)},~\Pi^{(t)} \geq 0, ~~t = 1, \dots, N, \\
&\hspace{1.5cm} {e}^{\top}_m{w} = 1, ~{w} \geq 0,
\end{aligned} 
\end{equation}
where $D^{(t)}$ denotes $\gamma_t\mathcal{D}(\mathcal{P}, \,\mathcal{P}^{(t)})$ for simplicity, {for $t = 1,\dots,N$}. Here, we assume that the dimensions of sampling distributions $a^{(t)}$ are equal to $n$ for convenience {and the case with different dimensions can be extended in a straightforward manner}. Since the last two constraints ${e}^{\top}_m{w} = 1, ~{w} \geq 0$ are redundant, we remove them and reformulate the problem \eqref{eq-wbp} as the following linear programming:
\begin{equation}\label{eq-wbp-lp}
\begin{array}{cl}
\min \limits_x \big\{  \langle c, x\rangle \;\mid\;
 {A} x={b},\; x \geq 0\big\},
\end{array}
\end{equation}
where 
\begin{itemize}
    \item $x=\left(\operatorname{vec}\left(\Pi^1\right);\dots ; \operatorname{vec}\left(\Pi^N\right) ; {w}\right)\in \mathbb{R}^{Nmn+m}$ ;
    \item $b=\left({a}^{(1)} ; {a}^{(2)} ; \dots ; {a}^{(N)} ; 0_{m} ; \dots ; 0_{m} \right)\in\mathbb{R}^{Nn+Nm}$;
    \item $c=\left(\operatorname{vec}\left(D^{(1)}\right) ; \dots ; \operatorname{vec}\left(D^{(N)}\right) ; 0_m\right)\in \mathbb{R}^{Nmn+m}$ ;
    \item $A=\left(\begin{array}{cc}
    A_1 & 0  \\
    A_2 & A_3 
    \end{array}\right)\in \mathbb{R}^{(Nn+Nm)\times (Nmn+m)}$, \,\,$A_1=\operatorname{Diag}\left(I_{n} \otimes {e}_m^{\top}, \dots, I_{n} \otimes {e}_m^{\top}\right)\in \mathbb{R}^{Nn\times Nmn},$\\
    
     $A_2=\operatorname{Diag}\left({e}_{n}^{\top} \otimes I_{m}, \dots, {e}_{n}^{\top} \otimes I_{m}\right)\in \mathbb{R}^{Nm\times Nmn}$, and $
     A_3=-{e}_N \otimes I_{m}\in \mathbb{R}^{Nm\times m}.$ 
\end{itemize}

\subsection{Newton equations}
Algorithm \ref{alg:smoothingNewton} can be directly {extended} to solve the WB problem \eqref{eq-wbp-lp}. Similar to {the linear system \eqref{eq-smoothing-update-y} for} OT problems, the most expensive step is to tackle a structured system of linear equations w.r.t. $\Delta y$ of the following form:
\begin{equation}\label{eq-wbp-deltay}
\left(\lambda I+A \Theta A^T\right) \Delta y=R,
\end{equation}
where $\lambda = \kappa_p \epsilon^k>0$ is a scalar, and
\begin{equation*}
    \begin{array}{ll}
        \Theta := \left(\left(1+\kappa_c \epsilon^k\right) I-V_2^k\right)^{-1} V_2^k&\in \mathbb{R}^{(Nm n +m) \times(Nm n +m)},\\[2pt]
        
        R:=r_p^k-A\left(\left(1+\kappa_c \epsilon^k\right) I-V_2^k\right)^{-1} r_c^k&\in\mathbb{R}^{Nm+Nn}.
    \end{array}
\end{equation*}
We aim to accelerate the computation by exploiting the special structure of the matrix $A$. {Recall that for the} OT problem, we can leverage the Fact \ref{fact-ADAT} to simplify the coefficient matrix of $\Delta y$ to \eqref{eq-ot-ADAT}. Here, for the WB problem, we use a similar technique to simplify the coefficient matrix in \eqref{eq-wbp-deltay} as follows. First, {we can verify that the matrix $\Theta$ is a nonnegative diagonal matrix.  In particular, $\Theta $ takes the following form:
\begin{align*}
    &\Theta:=\left(\begin{array}{cccc}
    \operatorname{Diag}(\theta^{(1)}) & & &\\
    & \ddots & &\\
    & & \operatorname{Diag}(\theta^{(N)}) & \\
    & & &\operatorname{Diag}(\bar\theta)
    \end{array}\right)
\end{align*}
where $\theta^{(t)} \in \mathbb{R}^{mn}_+$ for $t=1,\ldots,N$ and $\bar\theta\in \mathbb{R}^m_+$. Define $V_t := \mathrm{Mat}(\theta^{(t)})\in\mathbb{R}^{m\times n}$ for $t = 1,\dots, N$}. Then the system of linear equations \eqref{eq-wbp-deltay} can be equivalently written as
\begin{equation}\label{eq-wbp-y1y2}
    (\lambda I+A\Theta A^T) \Delta y=\left(\begin{array}{cc}
    E_1 & E_2 \\[2pt]
    E_2^T & E_3
    \end{array}\right)\left(\begin{array}{l}
    \Delta y_1 \\[2pt]
    \Delta y_2 
    \end{array}\right)=\left(\begin{array}{c}
    R_1 \\[2pt]
    R_2 
    \end{array}\right),
\end{equation}
where $\Delta y:=(\Delta y_1;\Delta y_2)\in \R^{Nn+Nm},R:=(R_1;R_2)\in\R^{Nn+Nm}$, and
\begin{equation*}
    \begin{array}{ll}
     E_1:={\operatorname{Diag}\left(\left(V_1^Te_m; \dots ; V_N^Te_m\right)\right)}+\lambda I&\in\R^{Nn\times Nn},\\[5pt]
     
   E_2:=\operatorname{Diag}\left({V}_1^T, \dots, {V}_N^T\right)&\in\R^{Nn \times Nm},\\[5pt]
   
   E_3:={\Diag\left(\left({V}_1e_n;\dots;{V}_Ne_n\right)\right)}+(e_Ne_N^T)\otimes
   \mathrm{Diag}( {\bar\theta})+\lambda I\quad&\in\R^{{Nm\times Nm}}.
\end{array}
\end{equation*}
To reduce the size of the linear system, we next eliminate $\Delta y_1$ and compute $\Delta y_2$ by
\begin{equation}\label{eq-wbp-y2}
    {S}\Delta y_2 = {R}_3,
\end{equation}
where ${R}_3:=R_2-E_2^TE_1^{-1}R_1\in\R^{Nm}$ and $S:=E_3-E_2^TE_1^{-1}E_2$ is the Schur complement matrix such that
\[
\arraycolsep=1.6pt\def\arraystretch{1.5}
    \begin{array}{lll}
        {S} &:= \Diag\left( {S}_1,\dots,{S}_N \right)
        {+\big(e_N\otimes \mathrm{Diag}(\sqrt{\bar\theta})\big) 
        \big(e_N \otimes \mathrm{Diag}(\sqrt{\bar\theta})\big)^\top}
        &\in\R^{Nm\times Nm},
        \\
        {S}_t&:={\Diag({V}_te_n+{\lambda e_m})
        -{V_t}\Diag\left(V_t^Te_m+{\lambda e_n}\right)^{-1}{V}_t^T} &\in\R^{m\times m},\; {t = 1,\dots, N} .
    \end{array}
\]
It is clear that the coefficient matrix $S$ is sparse and positive definite, so \eqref{eq-wbp-y2} can be efficiently solved directly by sparse Cholesky decomposition or iteratively by preconditioned conjugate gradient (PCG) method. Note that when forming the matrices $S_t$ for $t = 1,\dots, N$, matrix inversions are only applied to diagonal matrices. Therefore, the cost of computing the coefficient matrix is relatively modest. Finally, it follows from \eqref{eq-wbp-y1y2} that
\[
\Delta y_1=E_1^{-1}\left(R_1-E_2 \Delta y_2\right).
\]
We can {also compute $\Delta y_1$ efficiently} thanks to the diagonal block structure of $E_1$ and $E_2$. {To see this, we} denote $\Delta y_1:=(\Delta y_1^1;\dots;\Delta y_1^N)\in \R^{Nn}$, $\Delta y_2:=(\Delta y_2^1;\dots;\Delta y_2^N)\in \R^{Nm}$ and $R_1:=(R_1^1;\dots;R_1^N)\in\R^{Nn}$, then 
\begin{equation*}
    \Delta y_1^t =(R_1^t-{V}_t^T\Delta y_2^t )\oslash ({V}_t^Te_m+\lambda e_n),\quad t=1,\dots,N,
\end{equation*}
where ``$\oslash$'' is the element-wise division operator. Moreover, observe that the coefficient matrix in \eqref{eq-wbp-y2} is the sum of a symmetric positive definite and block-diagonal matrix and a low rank matrix. Thus, one can solve \eqref{eq-wbp-y2} efficiently by using the Sherman-Morrison-Woodbury formula. In our numerical experiments, we iteratively solve the equation \eqref{eq-wbp-y2} using the PCG method with incomplete Cholesky factorization as the preconditioner for the first few steps. We then switch to sparse Cholesky decomposition to solve it directly if the PCG steps exceed 80.

\section{Numerical Experiments}\label{sec:numexp}
In this section, we conduct a set of numerical experiments and present the corresponding computational results to demonstrate the efficiency of the proposed Algorithm \ref{alg:smoothingNewton}.

\subsection{Experimental settings}
We implement our algorithm purely in Julia (version 1.8.2) and compare it with the highly optimized commercial and/or open-source LP solvers including Gurobi (version 9.5.1), HiGHS (version 1.3.0) and CPLEX (version 22.1.0.0) by calling their Julia interfaces through JuMP~\footnote{Julia: \url{https://julialang.org/}; Gurobi.jl: \url{https://github.com/jump-dev/Gurobi.jl}; HiGHS.jl: \url{https://github.com/jump-dev/HiGHS.jl}; CPLEX.jl: \url{https://github.com/jump-dev/CPLEX.jl}; JuMP.jl: \url{https://github.com/jump-dev/JuMP.jl}. For the commercial solvers Gurobi and CPLEX, we use their academic licenses.}. In particular, for these powerful LP solvers, we compare our algorithm with their barrier methods and the network simplex method implemented in CPLEX~\footnote{Based on the benchmark in \cite{schrieber2016dotmark}, the network simplex method and its variants are the most efficient algorithms for solving OT problems to very high accuracy. However, only the network simplex method implemented in CPLEX has a friendly interface in JuMP, to the best of our knowledge. Therefore, we only compare Algorithm \ref{alg:smoothingNewton} with the network simplex method provided by CPLEX.}. Our experiments are run on a Linux PC having Intel Xeon E5-2680 (v3) cores with 96 GB of RAM.

Given an approximate KKT solution $(x,y,z)\in \mathbb{R}^{mn}\times \mathbb{R}^{m+n}\times \mathbb{R}^{mn}$, we define the maximal relative KKT residue as follows
\[
\eta_{\rm max}:=\max\left\{\eta_p:=\frac{\left\lVert Ax - d\right\rVert}{1+\left\lVert d\right\rVert}, \eta_d:=\frac{\left\lVert A^Ty + z-c\right\rVert}{1 + \left\lVert c\right\rVert}, \eta_c:=\frac{\left\lVert x - \Pi_+(x - z)\right\rVert}{1 + \left\lVert x\right\rVert + \left\lVert z\right\rVert}\right\}.
\]
Note that $\eta_d$ is always zero in our case based on our algorithmic design. To measure the duality gap, we also define the following relative gap
\[
	\eta_g:= \frac{\left\lvert \left\langle c, x\right\rangle - \left\langle d, y\right\rangle \right\rvert}{1+ \left\lvert \left\langle c, x\right\rangle \right\rvert + \left\lvert \left\langle d, y\right\rangle \right\rvert}.
\]
For a user-specified stopping tolerance $\tt{tol} > 0$, we terminate the algorithm when the smoothing parameter $\epsilon$ is smaller than $ \tt{tol}\times 10^{-2} $, or $\max\{\eta_{\rm max}, \eta_g\}\leq \tt{tol}$, or the maximum time $\tt{TimeLimit}$, or the maximum number of iterations $\tt{maxIter}$, is reached. In our implementation, we set $\tt{tol} = 10^{-8}$, {$\tt{TimeLimit}=24\,(hrs)$,} and $\tt{maxIter} = 1000$. For barrier methods and the network simplex method, we also set the related stopping tolerance as $\tt{tol}$ and compute the corresponding KKT residues $\eta_{\rm max}$ from the returned approximate solutions. For OT problems, we turn off the crossover phase for all LP solvers, as it is too time-consuming and in general does not improve the quality of the output solutions. On the other hand, we turn on the presolving phase for barrier methods, since it usually enhances the numerical stability and improves the convergence. However, for the network simplex method in CPLEX, we turn off the presolving phase since it consumes a lot of computational time but does not improve the performance of the network simplex method. Finally, we notice that the barrier method implemented in HiGHS only uses one thread when solving the problem. For WB problems, we turn off the presolving phase for the barrier method, since it does not improve the performance in our numerical tests. However, we turn on the presolving phase for dual simplex method since it can dramatically decrease the computational time based on our numerical experience. To ensure fair comparisons, we restrict all methods to use only one thread for OT problems, while allowing multi-threading for WB problems.

We also mention that for the matrix $A$ generated from an OT problem, the last row of $A$ is always redundant. Thus, to improve the practical performance of all the solvers, we always drop the last equality constraint when passing the input data to a solver. Also, we notice that performing a suitable scaling scheme to the problem data can improve the efficiency of Algorithm \ref{alg:smoothingNewton}. Therefore, in our implementation, for given input data $d$ and $c$, Algorithm \ref{alg:smoothingNewton} is executed on $\hat d$ and $\hat c$, where $\hat d = d / \left\lVert d\right\rVert$ and $\hat c = c / \left\lVert c \right\rVert$. However, the relative KKT residues and relative duality gap are evaluated on the original data. Finally, in our implementation, we set $\epsilon^0 = 1$, $r = 0.75$, $\tau = 0.25$, $\rho = 0.5$, $\mu = 10^{-8}$, $\sigma = \min\{10^3, \left\lVert c\right\rVert\}$, $\kappa_p = \kappa_c = 1$.

%%%%%%%%%%%%%%%%%%%%%%%%%%%%%%
\subsection{Computational results on DOTmark {for OT problems}}\label{subsec:DOTmark}

In our experiments, the cost matrix $C$ is chosen based on the squared Euclidean distance. In particular, for $1\leq i\leq m$ and $1\leq j\leq n$, $C_{ij} := (\ell_1 - \ell_2)^2 + (k_1 - k_2)^2$, where the indices $i$ and $j$ correspond to the $(\ell_1, k_1)$ and $(\ell_2, k_2)$ positions in the two images, respectively. Since $C$ contains some elements that are very large, we also normalize $ C $ by dividing it by its maximum value. The discrete probability distributions $a$ and $b$ are obtained from normalizing any two images. We consider images from the DOTmark collection \cite{schrieber2016dotmark} (a benchmark dataset for discrete OT problems) which contains ten classes of images arising from different scenarios. See Figure \ref{fig:dotmark} for examples of images from each class. Within a specific class that contains ten different images, we can pair any two different images and compute the optimal transport between them. Thus, each class of images generates 45 OT problems, resulting in 450 problems for ten classes. We mention that the DOTmark collection offers images at different resolutions, ranging from $32\times 32$ to $512\times 512$. However, due to limited available memory, we only present the results for the cases with $32\times 32$, $64\times 64$ resolutions and partial results for the cases with $ 128\times 128 $ resolution. As a consequence, $450 + 450 + 10 = 910$ OT problems are tested in our experiments. Table \ref{tab:sizes} summarizes the problem sizes with respect to different image resolutions. We see that for OT problems generated from $128 \times 128$ images, the problem sizes are typically too large to be handled by those LP solvers used on our machine.

\begin{figure}[H]
     \centering
     \begin{subfigure}[t]{0.15\textwidth}
         \centering
         \includegraphics[width=\textwidth]{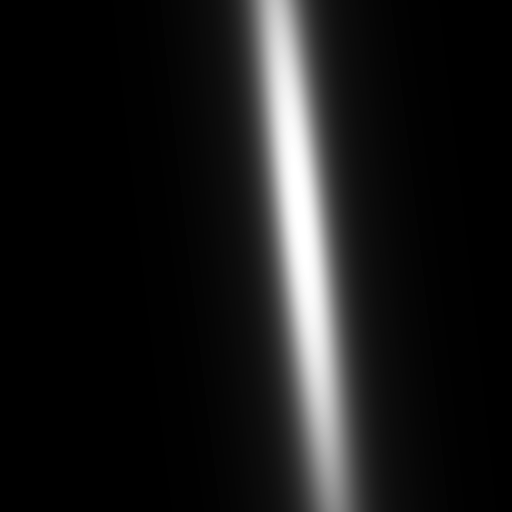}
         \caption{Cauchy}
         \label{fig:CauchyDensity}
     \end{subfigure}
     \begin{subfigure}[t]{0.15\textwidth}
         \centering
         \includegraphics[width=\textwidth]{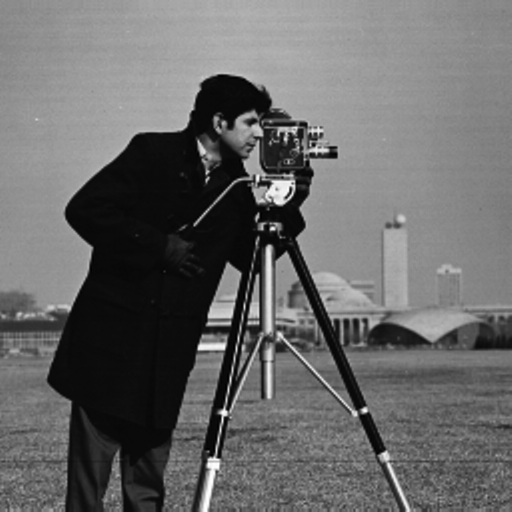}
         \caption{Classic}
         \label{fig:ClassicImages}
     \end{subfigure}
     \begin{subfigure}[t]{0.15\textwidth}
         \centering
         \includegraphics[width=\textwidth]{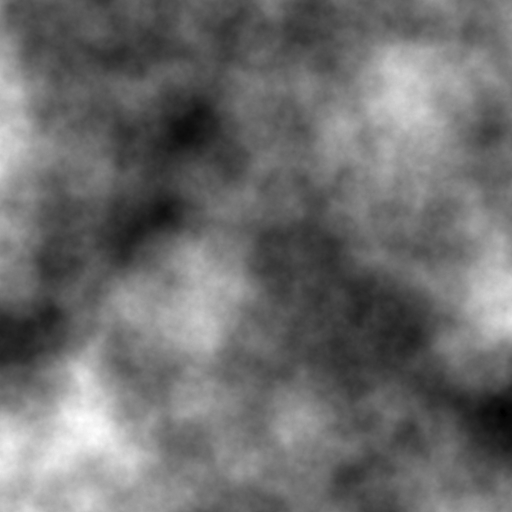}
         \caption{GRFm}
         \label{fig:GRFmoderate}
     \end{subfigure}
     \begin{subfigure}[t]{0.15\textwidth}
         \centering
         \includegraphics[width=\textwidth]{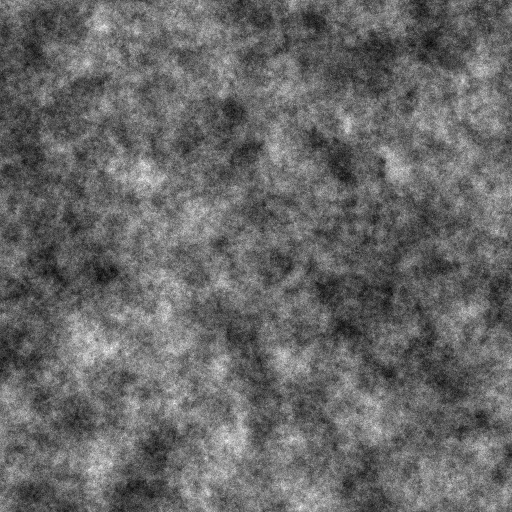}
         \caption{GRFr}
         \label{fig:GRFrough}
     \end{subfigure}
     \begin{subfigure}[t]{0.15\textwidth}
         \centering
         \includegraphics[width=\textwidth]{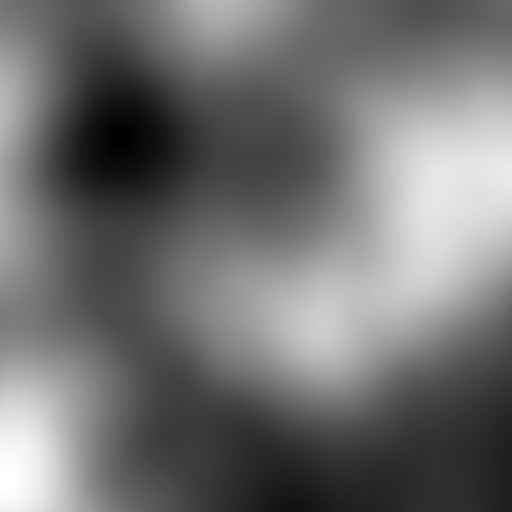}
         \caption{GRFs}
         \label{fig:GRFsmooth}
     \end{subfigure}
     
    \bigskip
   \begin{subfigure}[t]{0.15\textwidth}
         \centering
         \includegraphics[width=\textwidth]{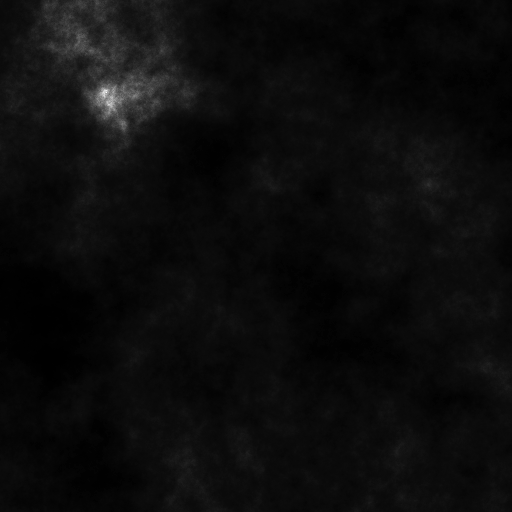}
         \caption{LogGRF}
         \label{fig:LogGRF}
     \end{subfigure}
     \begin{subfigure}[t]{0.15\textwidth}
         \centering
         \includegraphics[width=\textwidth]{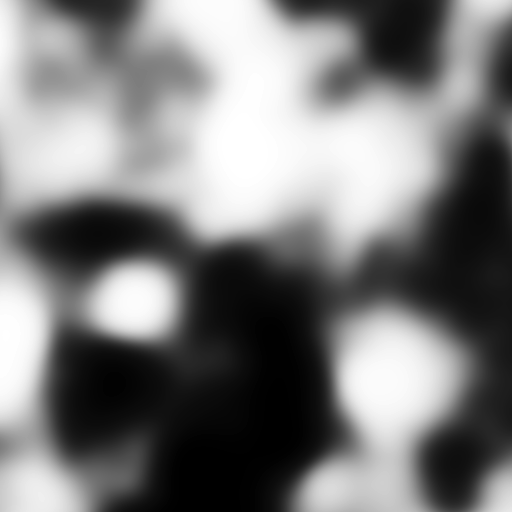}
         \caption{LogitGRF}
         \label{fig:LogitGRF}
     \end{subfigure}
     \begin{subfigure}[t]{0.15\textwidth}
         \centering
         \includegraphics[width=\textwidth]{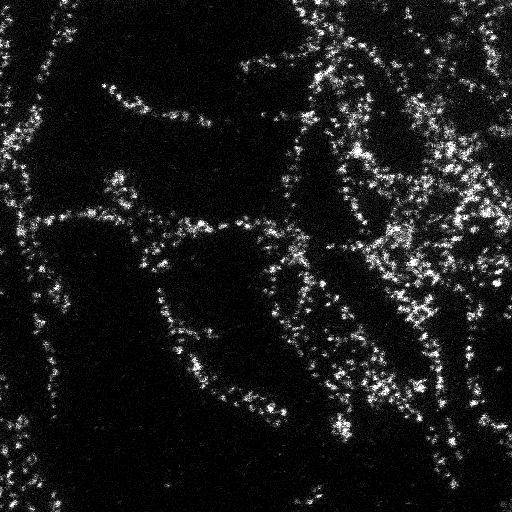}
         \caption{Microscopy}
         \label{fig:MicroscopyImages}
     \end{subfigure}
     \begin{subfigure}[t]{0.15\textwidth}
         \centering
         \includegraphics[width=\textwidth]{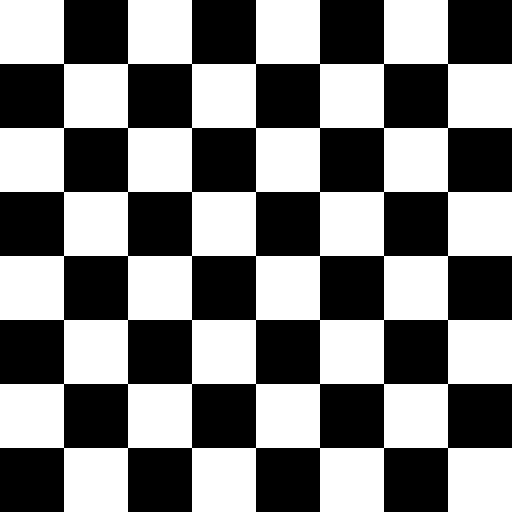}
         \caption{Shapes}
         \label{fig:Shapes}
     \end{subfigure}
     \begin{subfigure}[t]{0.15\textwidth}
         \centering
         \includegraphics[width=\textwidth]{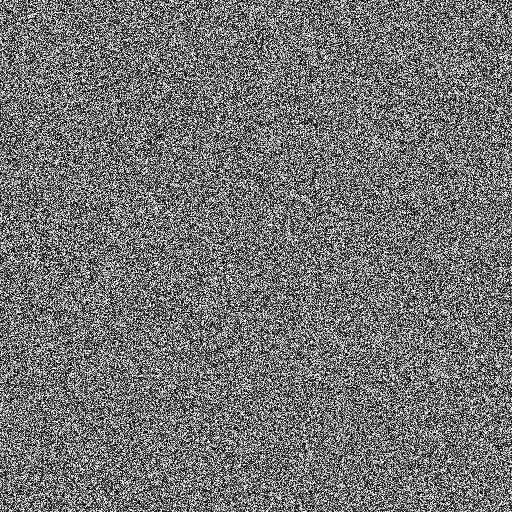}
         \caption{WhiteNoise}
         \label{fig:WhiteNoise}
     \end{subfigure}
    \caption{Example images from the DOTmark collection \cite{schrieber2016dotmark}.}
    \label{fig:dotmark}
\end{figure}

\begin{table}[H]
	\centering
	\begin{tabular}{|l|l|l|}
		\hline
		Resolution       & \#constraints & \#variables \\ \hline
		$32 \times 32$   & 2048         & 1,048,576    \\ \hline
		$64 \times 64$   & 8192         & 16,777,216   \\ \hline
		$128 \times 128$ & 32768        & 268,435,456  \\ \hline
	\end{tabular}
	\caption{OT problem sizes with different images resolutions.}
	\label{tab:sizes}
\end{table}

Notice that for most images from the ``Shapes'' and ``MicroscopyImages'' classes, the generated probability distributions may contain zero entries. In particular, if the $i$-th ($1\leq i\leq m$) element in $a$ is zero, then we can easily see that the $i$-th row of the optimal transportation plan $X$ is a zero vector. Similarly, if the $j$-th ($1\leq j\leq n$) element in $b$ is zero, then the $j$-th column of the optimal transportation plan $X$ is a zero vector. Therefore, one is able to drop those zero rows and/or columns in a prepossessing phase to get a smaller-scale OT problem.  As a consequence, we always solve the smaller OT problems. 

The detailed computational results for the instances with $64\times 64$ resolution are presented in Table~\ref{tab:img64}. (For brevity, we do not present the results for the instances with $32\times 32$ resolution but note that the relative performance of various methods are similar to those observed in Table~\ref{tab:img64}.) In the table, the barrier method and the network simplex method in CPLEX are denoted by ``CPLEX-Bar'' and ``CPLEX-Net'', respectively. For each image class, we present the mean values over 45 OT problems. We report the mean value of the number of iterations taken by each solver in the column ``Iter''. Note that CPLEX-Net performs the network simplex iterations, which are usually very large. Moreover, currently we do not know how to extract the total number of the network simplex iterations through JuMP. Therefore, we just set Iter to be '-' for CPLEX-Net. In the column ``Time'', the mean values of computational times for all the solvers are presented for comparing their practical efficiency. For the remaining four columns, we report the mean relative KKT residues (i.e., $ \eta_p $, $ \eta_d $, $ \eta_c $ and $ \eta_g $) in order to compare the accuracy of the solutions.

From the computational results, we observe that all the solvers except HiGHS were able to solve all the OT problems accurately. Indeed, the solutions provided by HiGHS are usually less accurate than the other solvers. Among all the solvers, CPLEX-Net showed the best performance, which coincides with the results presented in \cite{schrieber2016dotmark}. For the comparison between barrier methods, we observe that CPLEX-Bar and Gurobi had comparable performance. On the other hand, we see that HiGHS took more iterations and computational time than those of CPLEX-Bar and Gurobi. While Algorithm \ref{alg:smoothingNewton} took more iterations than those of the barrier methods, it is about 2--4 times faster than CPLEX-Bar and Gurobi. This indicates that Algorithm \ref{alg:smoothingNewton} is more efficient in terms of per-iteration cost when compared to barrier methods. This is exactly the consequence of the exploration of the solution sparsity.   

%%%%%%%%%%%%%%
\begin{footnotesize}
\begin{longtable}[c]{llrrrrrr}
\caption{Computational results for $64\times 64$ images. \label{tab:img64}} \\
\toprule
           Image &       Solver &  Time (s)&    Iter &  $ \eta_p $ &  $ \eta_d $ &   $ \eta_c $ &  $ \eta_g $ \\
\midrule
\endfirsthead

\multicolumn{8}{c}%
{{ Table \thetable\ continued from previous page}} \\
\toprule
           Image &       Solver &  Time (s) &    Iter &  $ \eta_p $ &  $ \eta_d $ &   $ \eta_c $ &  $ \eta_g $ \\
\midrule
\endhead
\midrule
\multicolumn{8}{r}{{Continued on next page}} \\
\midrule
\endfoot

\bottomrule
\endlastfoot
   \multirow{1}{*}{CauchyDensity} &        HiGHS &  5.0e+02 & 29 &  2.3e-06 &  8.2e-10 & 2.1e-10 & 1.8e-06 \\
   &       Gurobi &  3.0e+02 & 16 &  5.7e-11 &  1.6e-16 & 2.7e-09 & 2.2e-09 \\
   &    CPLEX-Bar &  6.0e+02 & 26 &  1.5e-10 &  3.4e-16 & 1.3e-10 & 1.5e-09 \\
   &    CPLEX-Net &  3.6e+01 & -  &  1.3e-18 &  6.2e-17 & 0.0e+00 & 7.6e-18 \\
   & SmoothNewton &  7.8e+01 & 57 &  3.1e-12 &  0.0e+00 & 3.9e-09 & 4.9e-09 \\[5pt]
   \multirow{1}{*}{GRFmoderate} &        HiGHS &  4.4e+02 & 27 &  1.1e-02 &  9.0e-02 & 8.4e-10 & 1.7e-02 \\
   &       Gurobi &  2.8e+02 & 15 &  2.5e-12 &  1.5e-16 & 1.6e-09 & 1.8e-09 \\
   &    CPLEX-Bar &  5.9e+02 & 25 &  5.2e-10 &  3.1e-15 & 3.4e-10 & 5.8e-09 \\
   &    CPLEX-Net &  3.3e+01 & -  &  6.8e-19 &  4.6e-17 & 0.0e+00 & 2.6e-18 \\
   & SmoothNewton &  7.1e+01 & 54 &  1.3e-12 &  0.0e+00 & 3.2e-09 & 4.9e-09 \\[5pt]
   \multirow{1}{*}{GRFsmooth} &        HiGHS &  4.6e+02 & 28 &  1.8e-06 &  1.5e-12 & 7.5e-10 & 1.1e-06 \\
   &       Gurobi &  3.0e+02 & 16 &  1.3e-11 &  1.6e-16 & 2.8e-09 & 2.9e-09 \\
   &    CPLEX-Bar &  6.2e+02 & 27 &  4.4e-10 &  3.0e-15 & 4.6e-10 & 4.8e-09 \\
   &    CPLEX-Net &  3.8e+01 & -  &  7.5e-19 &  5.6e-17 & 0.0e+00 & 3.1e-18 \\
   & SmoothNewton &  7.8e+01 & 58 &  1.9e-12 & 0.0e+00  & 3.0e-09 & 5.0e-09 \\[5pt]
   \multirow{1}{*}{LogitGRF} &        HiGHS &  5.2e+02 & 30 &  1.8e-05 &  1.2e-12 & 6.4e-10 & 2.2e-05 \\
   &       Gurobi &  3.0e+02 & 16 &  2.7e-12 &  1.5e-16 & 2.4e-09 & 2.6e-09 \\
   &    CPLEX-Bar &  6.8e+02 & 30 &  3.6e-10 &  2.4e-15 & 4.1e-10 & 2.9e-09 \\
   &    CPLEX-Net &  3.5e+01 & -  &  9.1e-19 &  5.0e-17 & 0.0e+00 & 2.9e-18 \\
   & SmoothNewton &  7.7e+01 & 59 &  1.3e-12 &  0.0e+00 & 3.3e-09 & 5.1e-09 \\[5pt]
   \multirow{1}{*}{Shapes} &        HiGHS &  9.2e+01 & 21 & 6.4e-07 & 1.6e-09 & 3.1e-10 & 7.0e-08 \\
   &       Gurobi &  6.2e+01 & 13 &  1.7e-11 &  9.2e-15 & 4.2e-10 & 1.8e-09 \\
   &    CPLEX-Bar &  1.0e+02 & 20 &  7.8e-10 &  9.6e-16 & 2.4e-10 & 4.5e-09 \\
   &    CPLEX-Net &  7.5e+00 & -  &  3.3e-18 &  3.8e-17 & 0.0e+00 & 8.2e-18 \\
   & SmoothNewton &  2.1e+01 & 52 &  6.3e-11 &  0.0e+00 & 8.8e-09 & 3.6e-09 \\[5pt]
   \multirow{1}{*}{ClassicImages} &        HiGHS &  5.1e+02 & 30 &  1.4e-06 &  1.0e-09 & 2.2e-10 & 2.5e-08 \\
   &       Gurobi &  3.0e+02 & 15 &  2.6e-12 &  1.4e-16 & 3.2e-09 & 2.5e-09 \\
   &    CPLEX-Bar &  5.9e+02 & 25 &  2.3e-10 &  2.0e-15 & 1.7e-10 & 2.4e-09 \\
   &    CPLEX-Net &  3.5e+01 & -  &  7.6e-19 &  4.1e-17 & 0.0e+00 & 2.0e-18 \\
   & SmoothNewton &  7.1e+01 & 55 &  1.2e-12 &  0.0e+00 & 3.2e-09 & 5.0e-09 \\[5pt]
   \multirow{1}{*}{GRFrough} &        HiGHS &  4.9e+02 & 29 &  5.5e-04 &  2.0e-02 & 5.6e-10 & 6.1e-04 \\
   &       Gurobi &  2.9e+02 & 15 &  1.5e-12 &  1.4e-16 & 1.8e-09 & 2.2e-09 \\
   &    CPLEX-Bar &  5.6e+02 & 24 &  2.5e-10 &  6.1e-15 & 5.0e-10 & 3.5e-09 \\
   &    CPLEX-Net &  3.4e+01 & -  &  7.3e-19 &  2.5e-17 & 0.0e+00 & 5.7e-19 \\
   & SmoothNewton &  7.4e+01 & 56 &  8.8e-13 &  0.0e+00 & 2.7e-09 & 4.9e-09 \\[5pt]
   \multirow{1}{*}{LogGRF}&        HiGHS &  9.1e+02 & 47 &  2.1e-03 &  5.5e-02 & 8.9e-10 & 5.3e-03 \\
   &       Gurobi &  3.1e+02 & 16 &  3.0e-12 &  1.7e-16 & 1.5e-09 & 2.3e-09 \\
   &    CPLEX-Bar &  5.9e+02 & 25 &  4.7e-10 &  1.9e-14 & 5.4e-10 & 6.1e-09 \\
   &    CPLEX-Net &  4.4e+01 & -  &  1.6e-18 &  6.3e-17 & 0.0e+00 & 8.7e-18 \\
   & SmoothNewton &  7.6e+01 & 58 &  3.2e-12 &  0.0e+00 & 3.8e-09 & 5.1e-09 \\[5pt]
   \multirow{1}{*}{MicroscopyImages} &        HiGHS &  1.1e+02 & 40 & 1.1e-06 & 3.0e-09 & 3.6e-10 & 8.7e-08 \\
   &       Gurobi &  3.7e+01 & 15 &  3.6e-12 &  1.5e-16 & 5.3e-09 & 1.2e-09 \\
   &    CPLEX-Bar &  6.1e+01 & 25 &  6.6e-10 &  1.1e-14 & 4.6e-10 & 3.4e-09 \\
   &    CPLEX-Net &  4.1e+00 & -  &  3.6e-18 &  5.1e-17 & 0.0e+00 & 6.4e-18\\
   & SmoothNewton &  9.8e+00 & 46 &  1.2e-10 &  0.0e+00 & 9.1e-09 & 2.0e-09 \\[5pt]
   \multirow{1}{*}{WhiteNoise}&        HiGHS &  4.8e+02 & 27 &  1.0e-02 &  9.5e-02 & 5.0e-10 & 6.4e-03 \\
   &       Gurobi &  3.0e+02 & 16 &  8.4e-13 &  1.4e-16 & 1.8e-09 & 1.9e-09 \\
   &    CPLEX-Bar &  6.1e+02 & 27 &  3.4e-10 &  4.3e-15 & 5.2e-10 & 4.5e-09 \\
   &    CPLEX-Net &  3.5e+01 & -  &  1.1e-18 &  1.6e-17 & 0.0e+00 & 2.8e-19 \\
   & SmoothNewton &  6.9e+01 & 56 &  7.9e-13 &  0.0e+00 & 3.3e-09 & 5.1e-09 \\
\end{longtable}

\end{footnotesize}

We also present some computational results for images with $128\times 128$ resolution. Since solving each OT problem of such a huge size is very time-consuming, we only conduct numerical experiments on the first two images in each class. Thus, only partial results are presented in Table \ref{tab:img128}. Note also that except Algorithm \ref{alg:smoothingNewton}, all other solvers usually take more than 96 GB of RAM to handle these OT problems, so they usually encounter out-of-memory issues. Hence, only the numerical results for Algorithm \ref{alg:smoothingNewton} are presented in Table~\ref{tab:img128}. We can see that the accuracy of the computed solutions is still very high. Moreover, it only took less than 40 minutes to converge, even though each problem contains more than 268 million nonnegative variables. This indicates that Algorithm \ref{alg:smoothingNewton} is quite efficient and robust. More importantly, it consumes less memory than state-of-the-art solvers. 

\begin{footnotesize}
\begin{longtable}[c]{lrrrrrr}
\caption{Partial computational results for Algorithm \ref{alg:smoothingNewton} on $128\times 128$ images. \label{tab:img128}}\\
\toprule
           Image &         Time (s) &    Iter &  $ \eta_p $ &  $ \eta_d $ &   $ \eta_c $ &  $ \eta_g $ \\
\midrule
\endfirsthead

\multicolumn{7}{c}%
{{ Table \thetable\ continued from previous page}} \\
\toprule
           Image &        CpuTime &    Iter &  $ \eta_p $ &  $ \eta_d $ &   $ \eta_c $ &  $ \eta_g $ \\
\midrule
\endhead
\midrule
\multicolumn{7}{r}{{Continued on next page}} \\
\midrule
\endfoot

\bottomrule
\endlastfoot 
		{CauchyDensity}    & 2.21e+03 &  96  & 1.1e-12 & 0.0e+00 & 2.8e-09 & 3.4e-09 \\[3pt] 
		{GRFmoderate}      & 2.17e+03 &  97  & 6.8e-13 & 0.0e+00 & 4.0e-10 & 5.4e-09 \\[3pt]  
		{GRFsmooth}        & 2.18e+03 &  93  & 8.7e-13 & 0.0e+00 & 3.7e-09 & 5.2e-09 \\[3pt]  
		{LogitGRF}         & 1.85e+03 &  86  & 8.0e-13 & 0.0e+00 & 4.1e-10 & 5.6e-09 \\[3pt] 
		{Shapes}           & 4.92e+02 &  48  & 3.1e-12 & 0.0e+00 & 2.6e-09 & 4.1e-09 \\[3pt] 
		{ClassicImages}    & 2.22e+03 &  101 & 6.4e-13 & 0.0e+00 & 4.3e-10 & 5.7e-09 \\[3pt]
		{GRFrough}         & 2.30e+03 &  97  & 7.6e-13 & 0.0e+00 & 4.4e-10 & 5.8e-09 \\[3pt]
		{LogGRF}           & 2.28e+03 &  108 & 9.2e-13 & 0.0e+00 & 4.1e-10 & 5.5e-09 \\[3pt]
		{MicroscopyImages} & 5.74e+02 &  88  & 1.4e-12 & 0.0e+00 & 1.2e-09 & 5.4e-09 \\[3pt]
		{WhiteNoise}       & 2.21e+03 &  99  & 8.0e-13 & 0.0e+00 & 4.4e-10 & 5.3e-09 \\
\end{longtable}
\end{footnotesize}

To further evaluate the memory consumption for each solver, we also measure the peak memory used by each solver when solving the problems. We use the program ``/usr/bin/time'' provided by the Linux system for our measurements. The output of the ``time'' program contains the term ``maximum resident set size'' (maximum RSS) which provides a reasonable estimate of the peak memory usage. Since the ``time'' program takes significant time for generating its output, we only measure the memory usage of each solver on the OT problems derived from the first two images in each image class. The results on memory usage are summarized in Table~\ref{tab:rss}. We can see that Algorithm \ref{alg:smoothingNewton} indeed used much less memory than the other solvers. This further shows the potential of Algorithm \ref{alg:smoothingNewton} for solving OT problems at scale. 

\begin{footnotesize}
\begin{longtable}[c]{llrrrrr}
\caption{Maximum RSS (MB) for the first two images in each class.}
\label{tab:rss}\\
\toprule
Image                          & Resolution     & HiGHS & Gurobi & CPLEX-Bar & CPLEX-Net & SmoothNewton \\ 
\midrule
\endfirsthead

\multicolumn{7}{c}%
{{ Table \thetable\ continued from previous page}} \\
\toprule
Image                          & Resolution     & HiGHS & Gurobi & CPLEX-Bar & CPLEX-Net & SmoothNewton \\ 
\midrule
\endhead
\midrule
\multicolumn{7}{r}{{Continued on next page}} \\
\midrule
\endfoot

\bottomrule
\endlastfoot 
\multirow{1}{*}{CauchyDensity} 
%& $32 \times 32$ &  1614     & 2005       &  1497         & 1620          &  726            \\  
                               & $64 \times 64$ & 22282      &   22604     &    19489       &  20224         &  3456            \\[0pt]
\multirow{1}{*}{GRFmoderate} 
%& $32 \times 32$ &  1615     &  1920      &   1556       &   1489        &   736           \\  
                               & $64 \times 64$ & 22281      & 22503      &  19500        &  20227         & 3751             \\[0pt]
\multirow{1}{*}{GRFsmooth} 
%& $32 \times 32$ & 1631      & 1921       &   1513        &   1527   & 726             \\  
                               & $64 \times 64$ & 22281      & 24152       &   19486        & 20225 &     3550         \\[0pt]
\multirow{1}{*}{LogitGRF} 
%& $32 \times 32$ & 1628      & 1921       &  1557         &  1527         &  741            \\ 
                               & $64 \times 64$ & 22331      & 21639       &  19492         & 20217          &   3850           \\[0pt]
\multirow{1}{*}{Shapes}
% & $32 \times 32$ &  1000     &  1089      &    1098       &  1054         &  647            \\  
                               & $64 \times 64$ & 7524      &  9236      &  9515         & 8924          &    2254          \\[0pt]
\multirow{1}{*}{ClassicImages} 
%& $32 \times 32$ & 1613      & 1926       & 1564          &  1530         & 735             \\  
                               & $64 \times 64$ & 22282      &  21668      &   19470        & 20228          & 3744             \\[0pt]
\multirow{1}{*}{GRFrough} 
%& $32 \times 32$ & 1616      & 1925       & 1559          &   1527      &   750           \\   
                               & $64 \times 64$ & 22281      &  24155      &   19499        &  20237         & 3541             \\[0pt]
\multirow{1}{*}{LogGRF} 
%& $32 \times 32$ & 1890      &  1923      & 1594          &   1604        &   769           \\   
                               & $64 \times 64$ &  22330     &  24158      &  19491         & 20226          &     3803         \\[0pt]
\multirow{1}{*}{MicroscopyImages} 
%& $32 \times 32$ &   1096    &  1396      &   1163        &    1116       &  654            \\   
                               & $64 \times 64$ &  2632     &  3049      & 2305          & 2471           &  919            \\[0pt]
\multirow{1}{*}{WhiteNoise} 
%& $32 \times 32$ &  1617     & 1926       &   1516        & 1543          &   723          \\   
                               & $64 \times 64$ & 22282      & 24158       &  19489         &  20229       &  3572       
                               \\[0pt]    
\end{longtable}
\end{footnotesize}

%%%%%%%%%%%%%%%%%%%%%%%%%%%%%
\subsection{Computational results on real data for WB problems}

In this subsection, we test our Algorithm \ref{alg:smoothingNewton} for computing WB of real image data. For a given set of images, we first rescale them to a pre-specified resolution, and vectorize and normalize them so that the sum of the resulting vectors are all ones. We generate the distance matrix $D^{(t)}$ for all $t=1,\dots N$ in the same manner as before. In our experiments, we set $m = n$ so that the barycenter has the same dimension as the
given images.

We conduct experiments on the MNIST dataset \cite{lecun1998gradient}, the Coil20 dataset \cite{nene1996columbia} and the Yale-Face dataset \cite{belhumeur1997eigenfaces}. For the MNIST dataset, we randomly select 10 images for each digit $(0,\dots,9)$ with the original resolution $28\times 28$. For the Coil20 dataset, we select 3 representative objects: Car, Duck and Pig. For each object, we choose 10 images and rescale them to $32\times 32$. For the Yale-Face dataset, we select three human subjects: Yale01, Yale02 and Yale03. For each subject, we randomly choose 10 images and rescale them to $30\times 40,36\times 48$ and $54\times 72$, respectively. See Figure \ref{fig:wb-examples} for an example of images from each class. A summary of the problem sizes is shown in Table \ref{tab:WB-sizes}. We set the weight $\gamma_t=1/N$ for all $t=1,\dots N$ and $p=2$. 

\begin{figure}[H]
    \centering
    \begin{subfigure}[t]{.15\textwidth}
        \includegraphics[width=\textwidth]{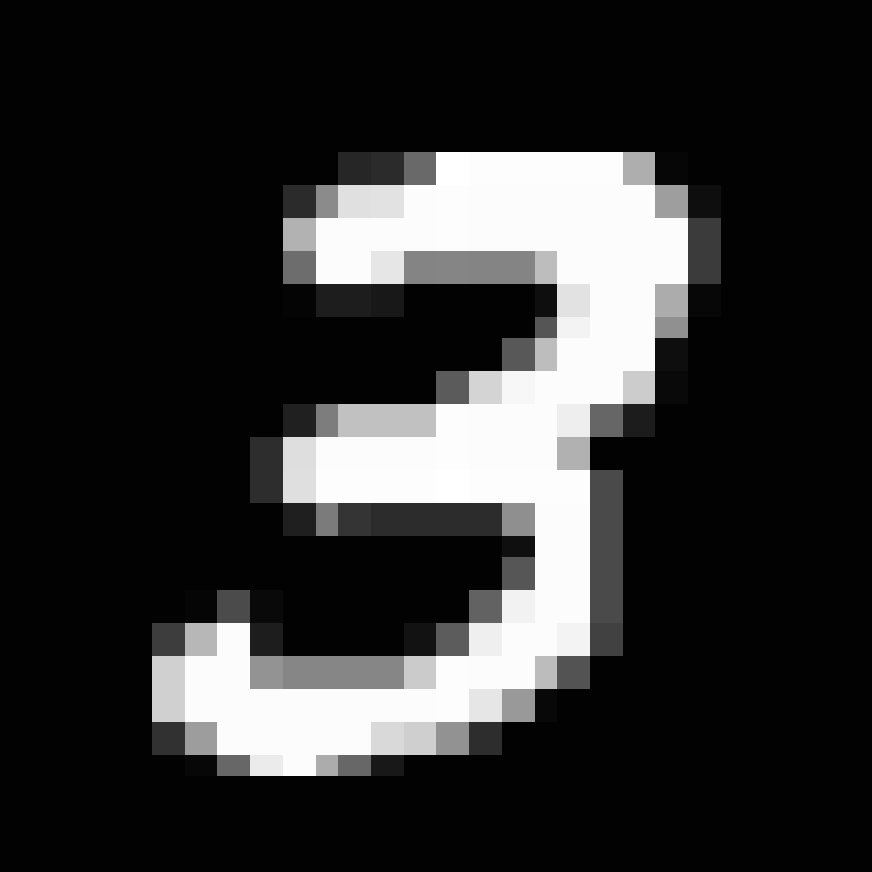}
        \caption{MNIST}
    \end{subfigure}
    \begin{subfigure}[t]{.15\textwidth}
        \includegraphics[width=\textwidth]{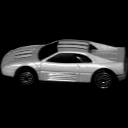}
        \caption{Car}
    \end{subfigure}
    \begin{subfigure}[t]{.15\textwidth}
        \includegraphics[width=\textwidth]{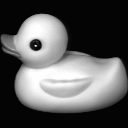}
        \caption{Duck}
    \end{subfigure}
    \begin{subfigure}[t]{.15\textwidth}
        \includegraphics[width=\textwidth]{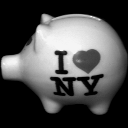}
        \caption{Pig}
    \end{subfigure}

    \bigskip
    \begin{subfigure}[t]{.2025\textwidth}
        \includegraphics[width=\textwidth]{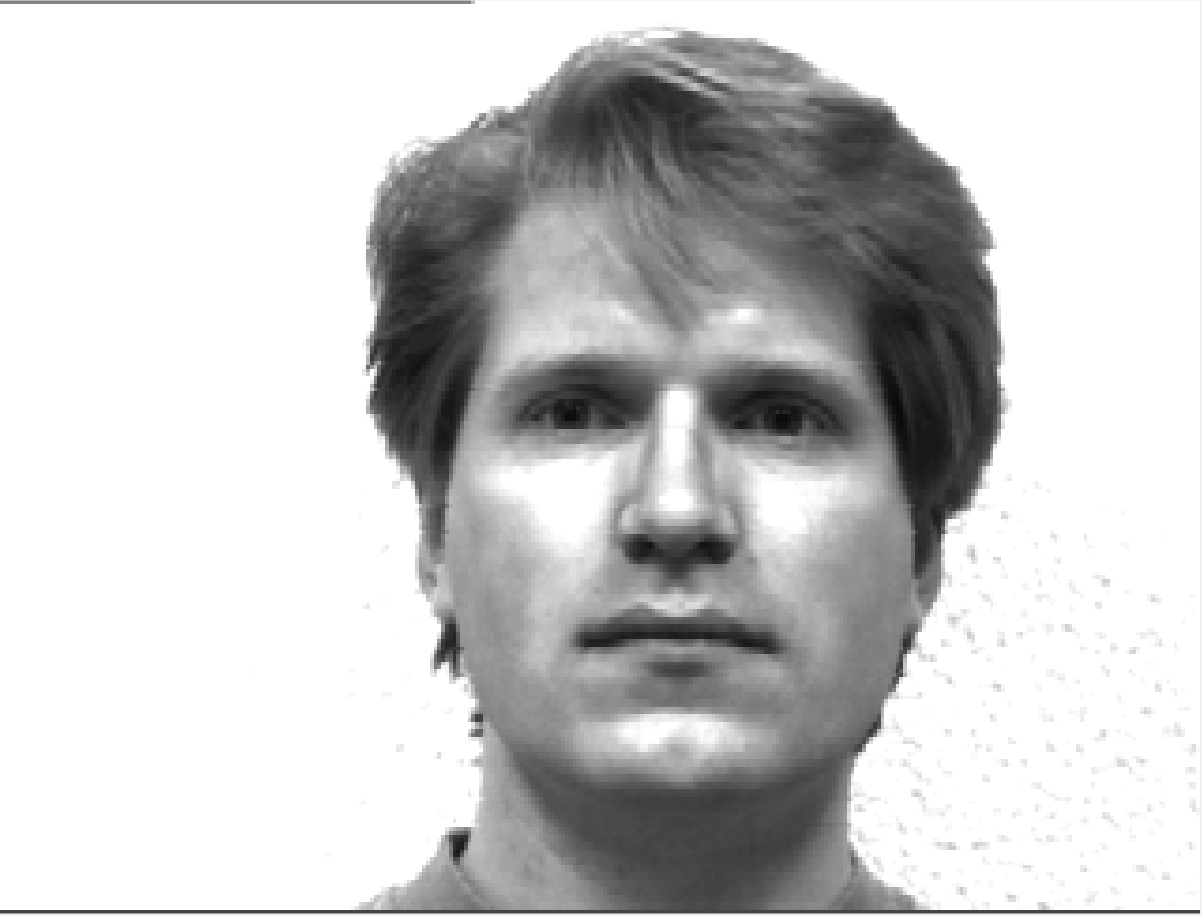}
        \caption{Yale01}
    \end{subfigure}
    \begin{subfigure}[t]{.2025\textwidth}
        \includegraphics[width=\textwidth]{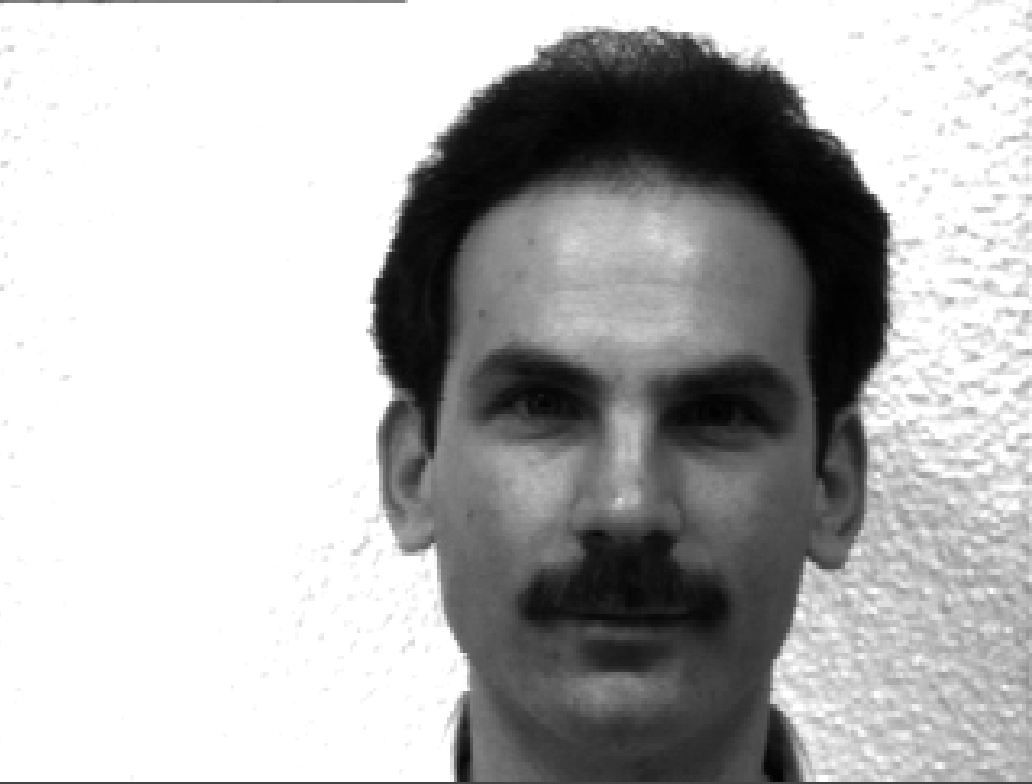}
        \caption{Yale02}
    \end{subfigure}
    \begin{subfigure}[t]{.2025\textwidth}
        \includegraphics[width=\textwidth]{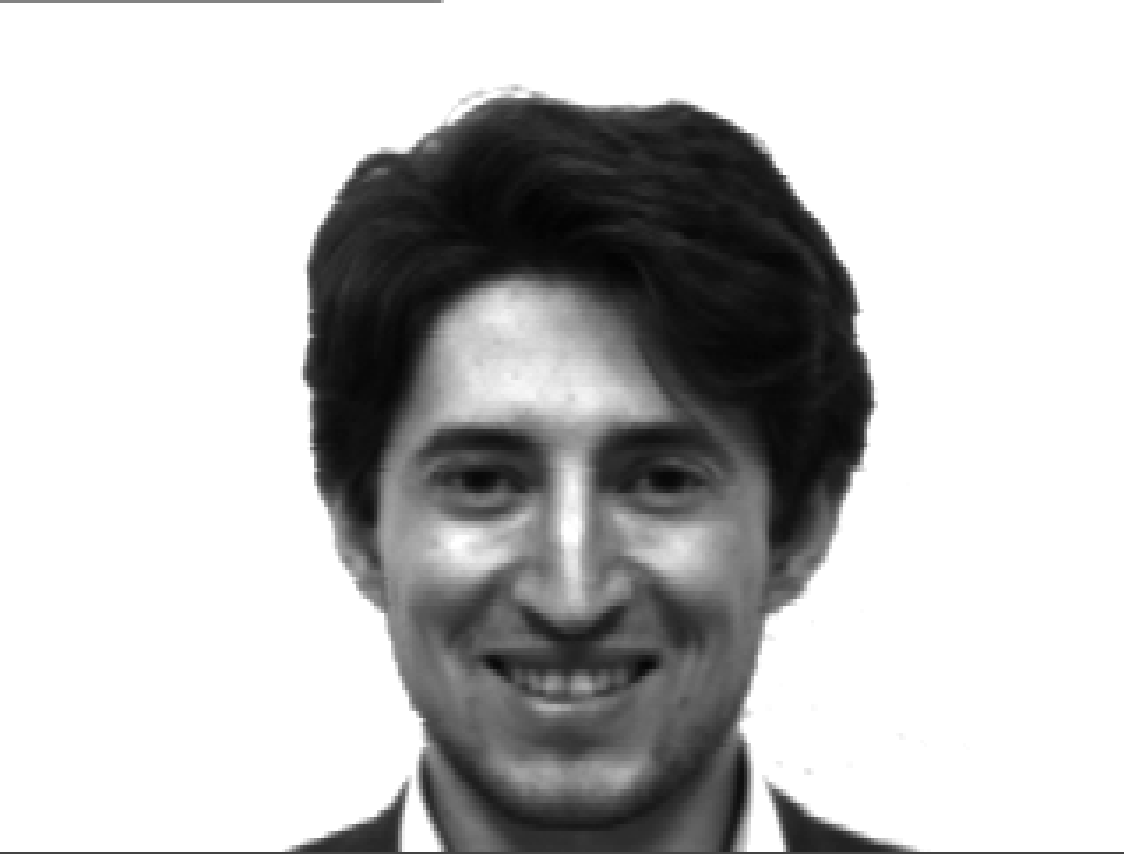}
        \caption{Yale03}
    \end{subfigure}
    \caption{Example images from different databases.}
    \label{fig:wb-examples}
\end{figure}

\begin{table}[H]
	\centering
	\begin{tabular}{|l|l|l|l|l|}
		\hline
		Database  & N  & Resolution  &\#constraints  & \#variables \\ \hline
		  MNIST     & 10 & $28 \times 28$  & 15680  & 6,147,344   \\ \hline
		  Car       & 10 & $32 \times 32$  & 20480  & 10,486,784  \\ \hline
		  Duck      & 10 & $32 \times 32$  & 20480  & 10,486,784  \\ \hline
            Pig       & 10 & $32 \times 32$  & 20480  & 10,486,784  \\ \hline
            Yale01    & 10 & $30 \times 40$  & 24000  & 14,401,200  \\ \hline
            Yale02    & 10 & $36 \times 48$  & 34560  & 29,861,568  \\ \hline
            Yale03    & 10 & $54 \times 72$  & 77760  & 151,169,328  \\ \hline
	\end{tabular}
	\caption{WB problem sizes with different image resolutions.}
	\label{tab:WB-sizes}
\end{table}

The results are summarized in Table \ref{tab:WBreal}. From the computational results, we observe that SmoothNewton and Gurobi-Bar are able to solve all the WB problems on image data successfully. However, Gurobi-Spx cannot solve {the largest} Yale-face problem within the 24-hour limit, since the corresponding LP is of very large size. On larger problems, Gurobi-Bar is faster than Gurobi-Spx. But SmoothNewton outperforms both Gurobi-Spx and Gurobi-Bar on all the image WB problems. For the WB instances on the Duck, Pig, and Yale-Face images, SmoothNewton is about ten times faster than both Gurobi-Spx and Gurobi-Bar.

\begin{footnotesize}
\begin{longtable}[c]{llrrrrrr}
\caption{Computational results on real data for {WB} problems. \label{tab:WBreal}} \\
\toprule
           Image & Solver &  Time(s) &  Iter & $\eta_p$ & $\eta_d$ & $\eta_c$ & $\eta_g$ \\
\midrule
\endfirsthead

\multicolumn{8}{c}%
{{ Table \thetable\ continued from previous page}} \\
\toprule
           Image & Solver &  Time(s) &  Iter & $\eta_p$ & $\eta_d$ & $\eta_c$ & $\eta_g$ \\
\midrule
\endhead
\midrule
\multicolumn{8}{r}{{Continued on next page}} \\
\midrule
\endfoot

\bottomrule
\endlastfoot
   \multirow{1}{*}
MNIST & Gurobi-Spx & 4.9e+01 & 103070 & 2.8e-16 & 8.1e-17 & 7.3e-19 & 3.5e-18 \\
   & Gurobi-Bar & 1.1e+02 & 30 & 5.0e-11 & 5.1e-15 & 6.1e-11 & 1.2e-09 \\
    & SmoothNewton & 1.3e+01 & 91 & 1.3e-13 & 0.0e+00 & 2.1e-10 & 8.2e-09 \\[2pt]

 \multirow{1}{*}
 Car & Gurobi-Spx & 1.2e+02 & 38608 & 1.6e-16 & 3.0e-17 & 2.0e-19 & 1.7e-18 \\
    & Gurobi-Bar & 2.2e+02 & 33 & 8.6e-12 & 8.4e-14 & 6.7e-13 & 6.1e-11 \\
    & SmoothNewton & 2.6e+01 & 111 & 6.2e-14 & 0.0e+00 & 1.6e-11 & 6.2e-10 \\[2pt]
 
 \multirow{1}{*}
 Duck & Gurobi-Spx & 3.2e+02 & 536738 & 2.6e-16 & 8.2e-17 & 1.8e-11 & 5.1e-18\\
    & Gurobi-Bar & 2.6e+02 & 43 & 3.0e-11 & 4.0e-15 & 1.6e-13 & 2.5e-12 \\
    & SmoothNewton & 2.6e+01 & 100 & 4.5e-13 & 0.0e+00 & 2.7e-10 & 9.1e-09\\[2pt] 
    
 \multirow{1}{*}
 Pig & Gurobi-Spx & 5.5e+02 & 1306882 & 1.6e-16 & 8.0e-17 & 4.2e-18 & 1.7e-17 \\
    & Gurobi-Bar & 2.4e+02 & 38 & 3.8e-11 & 5.2e-14 & 4.4e-12 & 7.4e-09 \\
    & SmoothNewton & 2.3e+01 & 96 & 1.2e-12 & 0.0e+00 & 1.8e-10 & 8.9e-09 \\[2pt]

 \multirow{1}{*}
 Yale01 & Gurobi-Spx & 3.7e+03 & 6126017 & 1.9e-16 & 8.7e-17 & 3.1e-17 & 1.9e-17 \\
    & Gurobi-Bar & 3.0e+02 & 32 & 3.6e-11 & 1.1e-15 & 4.1e-10 & 8.3e-09 \\
    & SmoothNewton & 3.1e+01 & 97 & 7.0e-13 & 0.0e+00 & 1.6e-10 & 6.6e-09 \\[2pt] 

 \multirow{1}{*}
 Yale02 & {Gurobi-Spx} & {2.0e+04} & 18892197 & 2.6e-16 & 8.5e-17 & 1.9e-11 & 2.5e-17 \\
    & Gurobi-Bar & {2.1e+03} & 45 & 1.4e-10 & 1.3e-13 & 2.2e-11 & 5.3e-09 \\
    & SmoothNewton & {1.8e+02} & 115 & 1.0e-12 & 0.0e+00 & 1.7e-10 & 9.7e-09 \\[2pt] 
    
 \multirow{1}{*}
 Yale03 & {Gurobi-Spx} & \multicolumn{2}{c}{{(exceed 24 hours)}} & - & - & - & - \\
    & {Gurobi-Bar} & {3.3e+04} & 26 & 1.9e-13 & 2.4e-12 & 4.1e-14 & 2.2e-09 \\
    & SmoothNewton & 2.5e+03 & 177 & 2.3e-13 & 0.0e+00 & 1.2e-10 & 8.6e-09\\[2pt] 
\end{longtable}
\end{footnotesize}

\section{Conclusions} \label{sec:conclusions}

In this paper, we have proposed a squared smoothing Newton method via the Huber smoothing function for solving the KKT system of the discrete optimal transport (OT) problem directly. Besides having appealing convergence properties, the proposed method is able to exploit the solution sparsity of the OT problem to greatly reduce the computational costs. We have also extended the method to solve WB problems. Our numerical experiments on solving a large set of OT and WB instances have demonstrated the excellent practical performance of the proposed method when compared to the state-of-the-art solvers. 

There are some issues concerning the proposed method that deserve further investigation. For example, the condition for ensuring the local superlinear convergence rate is rather strong and can not be verified a priori in general. Since we have observed the fast local convergence rate empirically based on our numerical tests, there may exist more relaxed conditions for ensuring such strong convergence properties, which deserve further studies. Moreover, it is clear that the method can also be applied to solving general linear programming (LP) problems. But the efficiency and robustness of the proposed method for solving general LP problems need further investigation. 

\section*{Acknowledgement}
The research of Kim-Chuan Toh is supported by the Ministry of Education, Singapore, under its Academic Research Fund Tier 3 grant call (MOE-2019-T3-1-010).

\bibliographystyle{unsrt}
\bibliography{sample-base}

\appendix

\section{Proofs}
\begin{proof}[{\bf Proof of Theorem \ref{thm-superlinear-rate}}]
	Since $ (\bar \epsilon,\bar x, \bar y) $ is an accumulation point of the sequence $ \{(\epsilon^k, x^k, y^k)\} $, Theorem \ref{thm-global-convergence} implies that $ \widehat{\mathcal{E}}(\bar\epsilon, \bar x, \bar y) = 0 $. In particular, it holds that $ \bar\epsilon = 0 $ and $ \zeta_k 
	= r\lVert\widehat{\mathcal{E}}(\epsilon^k, x^k, y^k)\rVert^{1+\tau} $ for $ k\geq 0 $ sufficiently large. By \cite[Proposition 3.1]{qi1993nonsmooth} and the fact that $ \epsilon^k > 0 $, the nonsingularity of the set $ \partial \widehat{\mathcal{E}}(\bar\epsilon, \bar x, \bar y)  $ implies (see e.g. \cite[Proposition 3.1]{qi1993nonsmooth}) that 
	\begin{equation} \label{eq-inv-jacobian-bdd}
		\left\lVert \widehat{\mathcal{E}}'(\epsilon^k, x^k, y^k)^{-1}\right\rVert = O(1),
	\end{equation}
	for $ k\geq 0 $ sufficiently large. 
 For notational simplicity, we let
 $w^k = (\epsilon^k; x^k; y^k)$,
 $\bar{w} = (\bar{\epsilon}; \bar{x}; \bar{y})$,
 and $\Delta w^k = (\Delta\epsilon^k; \Delta x^k; \Delta y^{k})$. 
 Using \eqref{eq-inv-jacobian-bdd}, we see that 
	\begin{equation} \label{eq-thm-rate-1}
		\begin{aligned}
			&\; \norm{w^k+\Delta w^k - \bar{w}}
			= \;  
   \norm{w^k + \widehat{\mathcal{E}}'(w^k)^{-1}\left(\begin{pmatrix}
				\zeta_k\epsilon^0 \\ 0 \\ 0
			\end{pmatrix} - \widehat{\mathcal{E}}(w^k)\right) - 
			\bar{w}}\\
			=&\; \left\lVert-\widehat{\mathcal{E}}'(w^k)^{-1}\left(\widehat{\mathcal{E}}(w^k) - \widehat{\mathcal{E}}'(w^k)
             (w^k-\bar{w}) - 
    \begin{pmatrix}
					\zeta_k\epsilon^0 \\ 0 \\ 0
				\end{pmatrix}\right)  \right\rVert \\
			= & \; O\left(\norm{\widehat{\mathcal{E}}(w^k) - \widehat{\mathcal{E}}'(w^k)(w^k-\bar{w}) } \right) 
    + O\left(\zeta_k\epsilon^0\right) \\
			= & \; O\left(\norm{\widehat{\mathcal{E}}(w^k) - \widehat{\mathcal{E}}'(w^k) (w^k-\bar{w})} \right) + O\left(\left\lVert\widehat{\mathcal{E}}(\epsilon^k, x^k, y^k)\right\rVert^{1+\tau}\right).
		\end{aligned} 
	\end{equation}
	Since $ \widehat{\mathcal{E}} $ is strongly semismooth everywhere, we have that, for $ k\geq 0 $ sufficiently large,
	\begin{equation} \label{eq-thm-rate-2}
 \norm{\widehat{\mathcal{E}}(w^k) - \widehat{\mathcal{E}}'(w^k) (w^k-\bar{w})}
 = 
				O\left(\norm{w^k-\bar{w}}^2 \right) .
	\end{equation}
	Moreover, since $ \widehat{\mathcal{E}} $ is locally Lipchitz continuous at $ \bar{w} $, it follows that 
	\begin{equation} \label{eq-thm-rate-3}
		\left\lVert\widehat{\mathcal{E}}(w^k)\right\rVert^{1+\tau} = \left\lVert\widehat{\mathcal{E}}(w^k) - \widehat{\mathcal{E}}(\bar{w})\right\rVert^{1+\tau} = O\left(\norm{w^k-\bar{w}}^{1+\tau}\right).
	\end{equation}
	
	By combining \eqref{eq-thm-rate-1}, \eqref{eq-thm-rate-2} and \eqref{eq-thm-rate-3}, we get for $k\geq 0$ sufficiently large, 
	\begin{equation} \label{eq-thm-rate-4}
	\begin{aligned}
		 	\norm{w^k+\Delta w^k -\bar{w}} = &\; O\left(\norm{w^k-\bar{w}}^{2}\right) + O\left(\norm{w^k-\bar{w}}^{1+\tau}\right) \;
				=\;  O\left(\norm{w^k-\bar{w}}^{1+\tau}\right).
		 \end{aligned}
	\end{equation}
	
	On the other hand, by \eqref{eq-inv-jacobian-bdd}, we can conclude that 
 $ \big\|(\widehat{\mathcal{E}}'(w^k))^{-1}\big\|$ is bounded away from zero for $ k\geq 0 $ sufficiently large. Therefore, it follows from \eqref{eq-thm-rate-2} that, for $ k\geq 0 $ sufficiently large,
	\begin{equation} \label{eq-thm-rate-5}
		\norm{w^k-\bar{w}} = O\left(\left\lVert\widehat{\mathcal{E}}(\epsilon^k, x^k, y^k) \right\rVert\right).
	\end{equation}

	Using \eqref{eq-thm-rate-4}, \eqref{eq-thm-rate-5} and the fact that $ \phi  $ and $ \widehat{\mathcal{E}} $ are both locally Lipschitz continuous at $ (\bar \epsilon, \bar x, \bar y) $, we get for $ k\geq 0 $ sufficiently large that
		\begin{align*} 
			 &\:\phi( w^k + \Delta w^k) = \; \left\lVert \widehat{\mathcal{E}}(w^k + \Delta w^k) - \widehat{\mathcal{E}}(\bar w)\right\rVert^2 
			= \; O\left( \norm{w^k+\Delta w^k -\bar{w}}^2 \right) 
			= \; O\left(\norm{w^k-\bar{w}}^{2(1+\tau)}\right) \\
			= &\; O\left(\left\lVert\widehat{\mathcal{E}}(w^k) \right\rVert^{2(1+\tau)}\right) 
			= \; O\left(\phi(w^k)^{1+\tau}\right) 
			= \; o\left(\phi(w^k)\right).
		\end{align*}
	This shows that, for $ k\geq 0 $ sufficiently large, 
 $(\epsilon^{k+1},x^{k+1},y^{k+1}) = (\epsilon^k,x^k,y^k) +
 (\Delta \epsilon^k, \Delta x^k, \Delta y^k)$,
	i.e., the unit step size is eventually accepted. Therefore, the proof is completed.
\end{proof}

\end{document}